\documentclass[11pt,reqno]{amsart}
\usepackage{a4wide}
\usepackage{hyperref,amsmath,amsfonts,mathscinet,amssymb,amsthm}
\usepackage{accents}
\usepackage{dirtytalk}
\usepackage{paralist}
\usepackage{constants}
\usepackage{xcolor}
\usepackage{tikz-cd}
\usepackage{todonotes}
\usepackage[mathscr]{euscript}
\usepackage{enumitem}
\usepackage{mathtools}
\usepackage{etoolbox}
\usepackage{cite}

\usepackage{mathtools}

\theoremstyle{plain}
\newtheorem{thm}{Theorem}[section]
\newtheorem{lemma}[thm]{Lemma}
\newtheorem{cor}[thm]{Corollary}
\newtheorem{prop}[thm]{Proposition}

\theoremstyle{definition}
\newtheorem{rem}[thm]{Remark}

\newtheorem{example}[thm]{Example}
\newtheorem{defn}[thm]{Definition}

\expandafter\let\expandafter\oldproof\csname\string\proof\endcsname
\let\oldendproof\endproof
\renewenvironment{proof}[1][\proofname]{%
  \oldproof[\bf #1]%
}{\oldendproof}

\numberwithin{equation}{section}

\newcommand{\abs}[1]{\left|{#1}\right|}
\newcommand{\norm}[1]{\left\lVert{#1}\right\rVert}

\newcommand{\R}{\mathbb{R}}
\newcommand{\Co}{\mathbb{C}}
\newcommand{\N}{\mathbb{N}}
\newcommand{\M}{\mathcal{M}}

\DeclareMathOperator*{\loc}{loc}
\DeclareMathOperator*{\supp}{supp}
\DeclarePairedDelimiter\ceil{\lceil}{\rceil}

\newcommand{\glob}{\operatorname{glob}}
\newcommand{\GWL}{\mathscr{WL}}
\newcommand{\WL}{\mathit{WL}}
\newcommand{\fat}{\operatorname{Fat}}

\DeclareFontFamily{U}{matha}{\hyphenchar\font45}
\DeclareFontShape{U}{matha}{m}{n}{
      <5> <6> <7> <8> <9> <10> gen * matha
      <10.95> matha10 <12> <14.4> <17.28> <20.74> <24.88> matha12
      }{}
\DeclareSymbolFont{matha}{U}{matha}{m}{n}
\DeclareFontSubstitution{U}{matha}{m}{n}

\DeclareFontFamily{U}{mathx}{\hyphenchar\font45}
\DeclareFontShape{U}{mathx}{m}{n}{
      <5> <6> <7> <8> <9> <10>
      <10.95> <12> <14.4> <17.28> <20.74> <24.88>
      mathx10
      }{}
\DeclareSymbolFont{mathx}{U}{mathx}{m}{n}
\DeclareFontSubstitution{U}{mathx}{m}{n}

\DeclareMathDelimiter{\vvvert}{0}{matha}{"7E}{mathx}{"17}
\DeclarePairedDelimiterX{\normi}[1]
  {\vvvert}
  {\vvvert}
  {\ifblank{#1}{\:\cdot\:}{#1}}

\begin{document}

\title{Wiener--Luxemburg amalgam spaces revisited}

\author{Ivan Kotal\'ik}

\email[I.~Kotal\'ik]{ivan.kotalik690@student.cuni.cz}
\urladdr{0009-0009-0128-2111}

\address{
 Department of Mathematical Analysis,
	Faculty of Mathematics and Physics,
	Charles University,
	Sokolovsk\'a~83,
	186~75 Praha~8,
	Czech Republic}


\date{\today}

\maketitle

\section*{Abstract}

The fundamental reliance on nonincreasing rearrangements restricts the recently introduced Wiener--Luxemburg amalgam spaces only to the rearrangement-invariant setting. To overcome this barrier, we develop a novel comprehensive framework that expands these amalgam structures to arbitrary quasi-Banach function spaces. By constructing a new quasinorm based on measure-constrained suprema and infima, we successfully separate these amalgams from rearrangement theory. We resolve the ensuing structural challenges regarding the Fatou property in this generalized setting, show consistency with the original theory, and fully characterize associate spaces, continuous embeddings, and the absolute continuity of the norm.

\section{Introduction}
In various areas of mathematical analysis, the need has arisen to separately study the local blow-ups and the global decay of functions. Recently, the concept of Wiener--Luxemburg amalgam spaces has been proposed to satisfy this need in the context of Banach function spaces in~\cite{Amalgam}. However, this construction allows only for rearrangement-invariant spaces. For example, this poses problems if one wishes to work with spaces where the distribution of mass matters, e.g., weighted Lebesgue spaces. In this paper, we introduce a novel approach to these spaces, developing a comprehensive framework to overcome this limitation. Our proposed definition is inspired by the work of Z.~Mihula and M.~P\'{a}ndy~\cite{mihulapandy}, who utilize a similar construction to introduce a new notion of almost compactness \say{near infinity}. Consider two arbitrary quasi-Banach function spaces $A$, $B$. We redefine the Wiener--Luxemburg amalgam space as
\begin{equation*}
    \norm{f}_{\GWL(A,B)}:=\sup_{\mu(E)=1} \norm{f\chi_E}_A + \inf_{\mu(E)=1} \norm{f\chi_{R\setminus E}}_B,
\end{equation*}
where $f$ is a measurable function and $(R,\mu)$ is a $\sigma$-finite infinite measure space. We subsequently demonstrate that, under reasonable assumptions, this functional possesses the requisite properties. To establish our framework, we draw extensively upon recent advances in the theory of quasi-Banach function spaces. In particular, we rely on the two papers~\cite{MuNePeTu} and~\cite{Pesaquasi}. We believe that this theory may be utilized to obtain new insights about quasi-Banach function spaces. For example, the methods and techniques introduced in this paper have already been used in the closely related paper~\cite{pesa2026amalgamapproachcompactnessquasibanach} to provide new characterizations of compactness in quasi-Banach function spaces. \par
To fully appreciate the scope of this extension, we recall the original construction introduced in~\cite{Amalgam}. For two rearrangement-invariant Banach function spaces $A$, $B$ defined over the real half-line, the original Wiener--Luxemburg amalgam space is defined by the quasinorm
\begin{equation*}
    \norm{f}_{\WL(A,B)}:=\norm{f^*\chi_{[0,1)}}_A + \norm{f^*\chi_{[1,\infty)}}_B,
\end{equation*}
where $f^*$ denotes the nonincreasing rearrangement of a measurable function $f$. Although this formulation elegantly separates the local and global behavior of a function, its fundamental reliance on rearrangements limits the theory to rearrangement-invariant spaces. By replacing this rearrangement-based approach with suprema and infima over sets of unit measure, our framework not only recovers the classical definition in the rearrangement-invariant setting (as we will demonstrate) but successfully extends these concepts to broader contexts, such as weighted Lebesgue spaces, where nonincreasing rearrangements cannot be meaningfully applied.  \par
These spaces, along with their underlying theoretical concepts, have found broad application in various mathematical problems. They were used to make a counterexample regarding $p$-homogeneous spaces in~\cite[Proposition 2.11]{BOZA2026130469}. A.~Alberico, A.~Cianchi, L.~Pick, and L.~Slav\'ikov\'a have used them while analyzing optimal Sobolev-type embeddings across the whole Euclidean space within the framework of rearrangement-invariant Banach function spaces in~\cite{MR3809140}. J. Vyb\'iral has used them for the identification of the optimal target space for the limiting scenario of standard Sobolev embeddings over the full Euclidean space. A similar need arose while investigating generalized Lorentz--Zygmund spaces; this prompted the creation of broken logarithmic functions, enabling one to handle the local and global characteristics of functions independently. Readers seeking a thorough examination of these generalized Lorentz--Zygmund spaces should consult~\cite{MR1698383}. Additionally, this methodology has seen successful application within interpolation theory. In this area, characterizing the sums and intersections of spaces essentially as amalgams has demonstrated significant utility, as detailed in~\cite{MR3776083, MR576995, tomaszewski}. \par 
Historically, the concept of amalgam spaces can be traced all the way back to the now a century old paper~\cite{Wiener} by N.~Wiener. Since then, many different versions of amalgam spaces have been developed; notable examples include those introduced by F.~Holland~\cite{Holland} and H. G. Feichtinger~\cite{feichtinger1980banach}. Over recent decades, many iterations of these spaces have been utilized extensively across numerous fields as demonstrated in these articles ~\cite{MR788385, MR1062738, MR1152343, MR1503035, MR983891, MR435739, MR1145171, MR594411, MR116186, MR208277, MR602764}. \par

This paper has the following structure. In Section~\ref{section_prelims}, we provide a comprehensive summary of the necessary theory that is used throughout the rest of the paper. \par

In Section~\ref{weak_fatou_sec}, we deal with functionals that are almost a quasi-Banach function norm but do not have the strong Fatou property, which is usually referred to as just the Fatou property. We show that the situation is still salvageable if the functional at least has the weak Fatou property. Our main contribution is that we transfer existing results to the context of quasi-Banach function norms. These results are crucial for showing that our amalgams are quasi-Banach function spaces. \par

Section~\ref{amalgam_sec} serves as the main part of this text. We introduce the new definition for the Wiener--Luxemburg amalgam space. We investigate its basic properties, and we use the theory from Section~\ref{weak_fatou_sec} to show that our amalgam functional defines a quasi-Banach function space. Furthermore, we show that our definition is consistent with the original definition of the Wiener--Luxemburg amalgam space in the sense that if we work with rearrangement-invariant spaces, we obtain equivalent quasinorms. Next, we follow in the footsteps of~\cite{Amalgam} and study the associate space of our amalgam spaces, and, as a corollary, we find that our quasinorm produces a normable topology if we start with normed spaces. We fully characterize embeddings between our amalgam spaces, and we study extremal embeddings into the spaces $L^1$ and $L^\infty$. While the fundamental conclusion is the same as in~\cite{Amalgam}, the method to obtain it is significantly more involved. Moreover, we use these results to generalize the classical theorem that all rearrangement-invariant Banach function spaces lie in between the spaces $L^1\cap L^\infty$ and $L^1+L^\infty$ to general quasi-Banach function spaces. Lastly, we prove a characterization of the absolute continuity of the norm in our amalgam spaces and transfer the result back to the rearrangement-invariant amalgams. \par

\section{Preliminaries}\label{section_prelims}

This section establishes the basic theory and important results upon which we build the revised theory of Wiener--Luxemburg amalgam spaces. \par

We start with introducing some elementary notation. By $(R,\mu)$ or $(S,\nu)$ we denote an arbitrary (totally) $\sigma$-finite measure space. For a $\mu$-measurable set $E\subset R$, we denote its characteristic function as $\chi_E$. By $\M(R,\mu)$ we denote the set of all extended complex-valued $\mu$-measurable functions defined on $R$. In this set, we identify all functions that are equal $\mu$-almost everywhere. Furthermore, we denote the set of all nonnegative functions in $\M(R,\mu)$ as $\M_+(R,\mu)$ and the set of functions that are finite $\mu$-almost everywhere in $\M(R,\mu)$ as $\M_0(R,\mu)$. For better readability, we often shorten $\mu$-almost everywhere to just \say{a.e.}~when the measure used is clear from the context. In the particular case of the $n$-dimensional Lebesgue measure, we denote it as $\lambda^n$ or $\lambda$ in the one-dimensional case. \par

For two topological linear spaces $X$, $Y$, we denote by $X\hookrightarrow Y$ the fact that  $X\subset Y$ and the identity mapping $I\colon X \to Y$ is continuous, and we say that $X$ is continuously embedded into $Y$. \par
For an arbitrary function $f$ and some set $E$ in the domain of $f$, we denote the restriction of $f$ to $E$ as $f|_E$, and we denote the support of $f$ as $\supp f$, which we define as
\begin{equation*}
    \supp f := \left \{ x \in R: \left \lvert \overline{f} \right \rvert > 0 \right \},
\end{equation*}
where $\overline{f}$ is an arbitrary representative of $f$. Our arguments will at no point depend in any way on the choice of this representative. In particular, we always have $f = 0$ $\mu$-a.e.~on $R \setminus \supp f$.

\subsection{Norms and quasinorms}

In this subsection, we recall the definitions and essential properties of norms and quasinorms as they are the fundamental building blocks for everything that we study in this text.

\begin{defn}
    Let $X$ be a complex linear space. A functional ${\norm{\cdot}}\colon X \to [0,\infty)$ is called a \textsl{quasinorm} if it has the following properties for all $a\in \Co$ and $x,y\in X$:
    \begin{enumerate}[label=$(\roman*)$]
        \item $\norm{ax}=\abs{a}\norm{x}$,
        \item $\norm{x}=0$ if and only if $x=0$ in $X$,
        \item there exists a constant $C\geq 1$ such that $\norm{x+y}\leq C(\norm{x}+\norm{y})$.
    \end{enumerate}
    The least such constant in $(iii)$ is called the \textsl{modulus of concavity} of $\norm{\cdot}$, and we denote it as $C_{\norm{\cdot}}$ or $C_X$ if we associate the norm with some space $X$. If it is that $C_{\norm{\cdot}}=1$, we call the functional $\norm{\cdot}$ a \textsl{norm}.
\end{defn}

If two quasinorms are equivalent in the usual sense, we write $\norm{\cdot}_1\approx\norm{\cdot}_2$. To denote the norm of a bounded operator between quasinormed spaces  we write $\norm{T}_{X\to Y}$. \par

Obviously, a norm produces a metrizable topology on the linear space $X$. The case is not so simple when we speak about quasinorms. Luckily, the topology is metrizable due to the classical Aoki--Rolewicz theorem. For more information on this theorem and its proof, see~\cite[Theorem 5]{Aoki} or~\cite{Rolewicz}. \par

Now, we would like to mention a critical observation regarding the triangle inequality. It will allow us to handle sums within a quasinorm more easily. This observation has been used before in~\cite[Lemma 3.1]{Pesaquasi} and~\cite{maligrandatypecotype}.

\begin{prop}\label{nekvindas_trick}
    Let $\norm{\cdot}_X$ be a quasinorm on some complex linear space $X$. Let $x_n\in X$ be a sequence of points. Then
    \begin{equation*}
        \norm{\sum_{n=1}^N x_n}_X\leq \sum_{n=1}^N C_{\norm{\cdot}}^n \norm{x_n}_X.
    \end{equation*}
\end{prop}

Finally, we define the concepts of intersection and sum of spaces.
\begin{defn}
    Let $\norm{\cdot}_X$ and $\norm{\cdot}_Y$ be quasinorms on some complex linear spaces $X$ and $Y$. Suppose that $X$ and $Y$ are continuously embedded into some Hausdorff topological vector space $Z$. Then we define the spaces 
    \begin{equation*}
        \begin{aligned}
            X+Y&:=\{x+y: x\in X, y\in Y \}, \\
            X\cap Y &:=\{z: z\in X, z\in Y \},
        \end{aligned}
    \end{equation*}
    endowed with the quasinorms
    \begin{equation*}
        \begin{aligned}
            \norm{z}_{X+Y} &:= \inf_{\substack{x\in X, y\in Y\\z=x+y}} \norm{x}_X + \norm{y}_Y, \\
            \norm{z}_{X\cap Y} &:= \max\{ \norm{z}_X, \norm{z}_Y\}.
        \end{aligned}
    \end{equation*}
\end{defn}

One can easily verify that the defined functionals are indeed quasinorms.

\subsection{Banach function norms and quasinorms}

We now shift our focus towards a particular case of norms and quasinorms acting on spaces of functions. The functionals in question are the so-called Banach function norms and Banach function quasinorms. For more details on norms, we refer the reader to the standard book~\cite{BennetSharpley} and for quasinorms to the article~\cite{Pesaquasi}. For an alternative approach, one may consult~\cite{LORIST2024247}.

\begin{defn}
    Let ${\norm{\cdot}}\colon\M(R,\mu)\to [0,\infty]$ be a functional satisfying $\norm{\abs{f}}=\norm{f}$ for all $f\in\M(R,\mu)$. We say that $\norm{\cdot}$ is a \textsl{Banach function norm} if it satisfies the following for all $a\in \Co$, $f,f_n,g\in\M_+(R,\mu)$ and $E\subset R$ such that $f_n\nearrow f$ a.e.~and $\mu(E)<\infty$:
    \begin{enumerate}[label=(P\arabic*)]
        \item it is a norm, i.e.,~it satisfies
            \begin{enumerate}[label=(\alph*)]
                \item $\norm{af}=\abs{a}\norm{f}$,
                \item $\norm{f}=0$ if and only if $f=0$ a.e.,
                \item $\norm{f+g}\leq \norm{f}+ \norm{g}$,
            \end{enumerate}
        \item it has the lattice property, i.e.,~if $f\leq g$ a.e., then $\norm{f}\leq\norm{g}$,
        \item it has the (strong) Fatou property, i.e.,~$\norm{f_n}\nearrow\norm{f}$,
        \item it is nontrivial, i.e.,~$\norm{\chi_E}<\infty$,
        \item its elements are locally integrable, i.e.,~for the set $E$ there exists a constant $C_E>0$ such that $\int_E f \, d\mu \leq C_E \norm{f}$.
    \end{enumerate}
    We define the corresponding \textsl{Banach function space} as the collection 
    \begin{equation*}
        X:=\left\{f\in\M(R,\mu): \norm{f}_X < \infty \right\}.
    \end{equation*}
\end{defn}

Note that, as we will see later on, the strong Fatou property may be difficult or outright impossible to verify for some functionals. In Section~\ref{weak_fatou_sec}, we will introduce specific machinery in the form of the weak Fatou property to work around this issue. \par

As this set of properties may prove to be too restrictive, we introduce the much broader terms of quasi-Banach function norms and quasi-Banach function spaces.

\begin{defn}
    Let ${\norm{\cdot}}\colon\M(R,\mu)\to [0,\infty]$ be a functional satisfying $\norm{\abs{f}}=\norm{f}$ for all $f\in\M(R,\mu)$. We say that $\norm{\cdot}$ is a \textsl{quasi-Banach function norm} if it satisfies the axioms (P2), (P3), (P4) and a weaker version of (P1), namely
    \begin{enumerate}
        \item[(Q1)] it is a quasinorm, i.e.,~it satisfies
            \begin{enumerate}[label=(\alph*)]
                \item $\norm{af}=\abs{a}\norm{f}$,
                \item $\norm{f}=0$ if and only if $f=0$ a.e.,
                \item there exists a constant $C\geq1$ such that $\norm{f+g}\leq C(\norm{f}+ \norm{g})$.
            \end{enumerate}
    \end{enumerate}
    We define the corresponding \textsl{quasi-Banach function space} as the collection 
    \begin{equation*}
        X:=\left\{f\in\M(R,\mu): \norm{f}_X < \infty \right\}.
    \end{equation*}
\end{defn}
Again, we call the smallest possible constant in (Q1) c) the modulus of concavity, and we use the same notation as before. \par
Using the name \say{Banach} is justified as it has been proven in~\cite{Pesaquasi} that these spaces are indeed complete even in the quasinormed case. Another classical fact that has been extended to quasinorms in~\cite{Pesaquasi} is the following theorem.

\begin{thm}\label{continuous_embeddings}
    Let $\norm{\cdot}_X$ and $\norm{\cdot}_Y$ be quasi-Banach function norms and let $X$ and $Y$ be the corresponding quasi-Banach function spaces. Then $X\subset Y$ if and only if $X\hookrightarrow Y$.
\end{thm}

If needed, we can restrict ourselves to working only with functions that are finite almost everywhere in the context of quasi-Banach function spaces. We will be using this fact when needed, without explicit mention. This classical theorem, see~\cite[Chapter 1, Theorem 1.4]{BennetSharpley}, was shown to still be valid, even when considering a quasinorm, again in~\cite{Pesaquasi}.

\begin{thm}\label{local_measure_embedding}
    Let $\norm{\cdot}_X$ be a quasi-Banach function norm and let $X$ be the corresponding quasi-Banach function space. Then 
    \begin{equation*}
        X\hookrightarrow \M_0(R,\mu),
    \end{equation*}
    where we equip $\M_0(R,\mu)$ with the topology of convergence in measure on the sets of finite measure.
\end{thm}

As we know, the norms of characteristic functions of sets having the same measure may vary when working with quasi-Banach function spaces. However, the following theorem shows that we are still able to say something about these norms if the underlying measure space is nonatomic. This result was first shown in~\cite[Lemma 4.4]{Slavikova2012}. We state this theorem in a slightly more general manner, but the original proof still works without any alterations. Also note that the first part of the statement is trivial in the case of a finite measure space due to the property (P4).

\begin{thm}\label{lenka4.4}
    Let $(R,\mu)$ be a nonatomic measure space, $\norm{\cdot}_X$ be a quasi-Banach function norm and let $X$ be the corresponding quasi-Banach function space. Let $t\in(0,\infty)$. Then there exists a constant $C_1>0$ such that for any $E\subset R$ satisfying $\mu(E)=t$, we have 
    \begin{equation*}
        \norm{\chi_E}_X \leq C_1.
    \end{equation*}
    Furthermore, if $\norm{\cdot}_X$ is a Banach function norm, there exists a constant $C_2>0$ such that for any $E\subset R$ satisfying $\mu(E)=t$, we have 
    \begin{equation*}
        \norm{\chi_E}_X \geq C_2.
    \end{equation*}
\end{thm}

 This fact was recently noticed again in the article~\cite{mihulapandy} and prompted the authors to introduce a new concept of extremal fundamental functions defined below.
\begin{defn}
    Let $\norm{\cdot}_X$ be a quasi-Banach function norm and let $X$ be the corresponding quasi-Banach function space. We define the \textsl{minimal fundamental function} of $\norm{\cdot}_X$ as 
     \begin{equation*}
        \varphi_X^{\min}(t) := \inf_{\substack{E\subset R \\ \mu(E)=t}} \norm{\chi_E}_X, \ \text{for } t \text{ in the range of } \mu,
     \end{equation*}
     and the \textsl{maximal fundamental function} of $\norm{\cdot}_X$ as 
     \begin{equation*}
        \varphi_X^{\max}(t) := \sup_{\substack{E\subset R \\\mu(E)=t}} \norm{\chi_E}_X, \ \text{for } t \text{ in the range of } \mu.
     \end{equation*}
\end{defn}

Notice that the maximal fundamental function is nonzero, for every $t>0$, because if it were not, then $\norm{\chi_E}_X=0$ for some set of positive measure. However, that is in contradiction to the property (Q1) of $\norm{\cdot}_X$. Also note that due to Theorem~\ref{lenka4.4} the maximal fundamental function is always finite and that for Banach function norms the minimal fundamental function is nonzero for $t>0$ if our measure space is nonatomic. For quasinorms, that may not be the case, as demonstrated by the following example.

\begin{example}\label{non_admissibility}
    Let $w(x):=\min\{1,1/x^3\}$. Consider the quasi-Banach function space $L^1(w)([0,\infty),\lambda)$ endowed with the norm
    \begin{equation*}
        \norm{f}_{L^1(w)}:=\int_0^\infty \abs{f(x)}w(x) \, dx, \ f\in \M([0,\infty),\lambda).
    \end{equation*}
    It is easy to check that this functional satisfies all axioms (P1)-(P4). We calculate that
    \begin{equation*}
        \norm{\chi_{(n,n+1)}}_{L^1(w)}=\int_n^{n+1} \min\left\{1,\frac{1}{x^3}\right\} \, dx \leq \frac{1}{n^3}\to 0.
    \end{equation*}
\end{example}

It will be useful to restrict ourselves to spaces that possess the property that their minimal fundamental function is nonzero for $t>0$.

\begin{defn}
    Let $\norm{\cdot}_X$ be a quasi-Banach function norm and let $X$ be the corresponding quasi-Banach function space. 
    We say that $\norm{\cdot}_X$ is \textsl{admissible} if $\varphi_X^{\min}(t)>0$ for all $t>0$.
\end{defn}

The next proposition provides a sufficient condition for admissibility, which is a generalization of the second part of Theorem~\ref{lenka4.4}. This was observed in~\cite[Proposition 3.2]{mihulapandy}.

\begin{prop}\label{p5admissible}
    Let $(R,\mu)$ be a nonatomic measure space and let $\norm{\cdot}_X$ be a quasi-Banach function norm with the property (P5). Then $\norm{\cdot}_X$ is admissible.
\end{prop}

Another property of these spaces that we would like to investigate is the absolute continuity of the norm, which is studied in~\cite[Chapter 1.3]{BennetSharpley}. 

\begin{defn}
     Let $\norm{\cdot}_X$ be a quasi-Banach function norm and let $X$ be the corresponding quasi-Banach function space. A function $f\in\M(R,\mu)$ is said to have an \textsl{absolutely continuous norm} in $X$ if for every sequence $E_n\subset R$ satisfying $\chi_{E_n}\to 0$ a.e., it holds $\norm{f\chi_{E_n}}\to 0$. We denote the subspace of all functions having absolutely continuous norm in $X$ as $X_a$. If it is that $X=X_a$, we say that the space $X$ has an \textsl{absolutely continuous norm}. 
\end{defn}

\subsection{Associate space}

Associate spaces are an important notion in the theory of Banach function spaces and quasi-Banach function spaces. For a thorough study of them in the case of Banach function spaces, we refer the reader to~\cite[Chapter 1, Sections 2, 3 and 4]{BennetSharpley}. \par
Notice that the definition of the associate space does not require any special assumptions on the original functional.

\begin{defn}
    Let ${\norm{\cdot}_X}\colon \M(R,\mu) \to [0,\infty]$ be a nonnegative functional, and let
    \begin{equation*}
        X:=\left\{f\in\M(R,\mu): \norm{f}_X < \infty \right\}.
    \end{equation*}
    We define the \textsl{associate functional} of $\norm{\cdot}_X$ as 
    \begin{equation*}
        \norm{f}_{X'}:=\sup\left\{\int_R \abs{fg} \, d\mu: g\in X \text{,  } \norm{g}_X\leq 1\right\},
    \end{equation*}
    for any $f\in \M(R,\mu)$. Furthermore, we define the \textsl{associate space} of $X$ as the collection 
    \begin{equation*}
        X':=\left\{f\in\M(R,\mu): \norm{f}_{X'} < \infty \right\}.
    \end{equation*}
\end{defn}

Note that we will usually only write
\begin{equation*}
    \norm{f}_{X'}= \sup_{\norm{g}_X\leq 1} \int_R \abs{fg} \, d\mu.
\end{equation*}
We are mostly interested in the particular case, where we work with at least a quasinorm. Nevertheless, we think it is interesting to mention that almost no assumptions are required. In fact, the next classical result, which is the famous H\"older inequality, also works without any additional assumptions.
\begin{thm}\label{holder}
    Let ${\norm{\cdot}_X}\colon \M(R,\mu) \to [0,\infty]$ be a positively homogeneous nonnegative functional and $\norm{\cdot}_{X'}$ be its associate functional. Then for all $f\in \M(R,\mu)$, it holds that
    \begin{equation*}
        \int_R \abs{fg} \, d\mu \leq \norm{g}_X \norm{f}_{X'},
    \end{equation*}
    where $0\cdot\infty=-\infty \cdot \infty = \infty$.
\end{thm}

The convention at the end of the previous theorem is needed because our lack of assumptions allows for pathological cases that require special attention. More precisely, we allow that $\norm{g}_X=0$ even if $g\not= 0$ a.e., and it can also happen that $\norm{f}_{X'}=\sup \emptyset=-\infty$. \par
From now on, we assume that $\norm{\cdot}_X$ is a quasi-Banach function norm. Note that in this case, it is easily verified that in the definition of the associate norm, we can only work with functions in the unit sphere, that is,
\begin{equation*}
    \norm{f}_{X'}= \sup_{\norm{g}_X= 1} \int_R \abs{fg} \, d\mu.
\end{equation*}\par 
We follow with a theorem due to A.~Gogatishvili and F.~Soudsk\'y from~\cite{SoGo} that generalizes the classical result by G.~G.~Lorentz and W.~A.~J.~Luxemburg, which can be found in~\cite[Chapter 1, Theorem 2.7]{BennetSharpley}. It shows that we need very little for the associate functional to be a Banach function norm. In particular, when a quasi-Banach function norm satisfies (P5), its associate functional is a Banach function norm. It also includes the characteristic property of Banach function spaces that they coincide with their second associate space.

\begin{thm}\label{second_dual}
    Let ${\norm{\cdot}_X}\colon \M(R,\mu) \to [0,\infty]$ be a nonnegative functional that satisfies (P4), (P5), and for all $f\in\M(R,\mu)$ it holds that $\norm{f}_X=\norm{\abs{f}}_X$. Then its associate functional $\norm{\cdot}_{X'}$ is a Banach function norm. \par 
    Furthermore, $\norm{\cdot}_X$ is equivalent to a Banach function norm if and only if $\norm{\cdot}_X\approx \norm{\cdot}_{X''}$, where $\norm{\cdot}_{X''}$ denotes the associate functional of $\norm{\cdot}_{X'}$, and if $\norm{\cdot}_X$ is a Banach function norm, then $\norm{\cdot}_X=\norm{\cdot}_{X''}$.
\end{thm}

The last result regarding associate spaces that we need is the following version of Landau's resonance theorem, which was first proven in this generality in~\cite{Pesaquasi}.

\begin{thm}\label{associate_char}
    Let $\norm{\cdot}_X$ be a quasi-Banach function norm, let $X$ be the corresponding quasi-Banach function space, and let $\norm{\cdot}_{X'}$ and $X'$, respectively, be the associate norm of $\norm{\cdot}_X$ and the corresponding associate space. Then for $f\in \M(R,\mu)$ we have $f\in X'$ if and only if it satisfies 
    \begin{equation*}
        \int_R \abs{fg} \, d\mu < \infty,
    \end{equation*}
    for all $g\in X$ such that $\norm{g}_X\leq 1$. \par
    In particular, if $f\notin X'$, then there exists $g\in X$ satisfying $\norm{g}_X\leq 1$ such that 
    \begin{equation*}
        \int_R \abs{fg} \, d\mu = \infty.
    \end{equation*}
\end{thm}

\subsection{Nonincreasing rearrangement}

Although the goal of this text is to alleviate the need for rearrangements, we will now turn our attention to the nonincreasing rearrangement and rearrangement-invariant spaces because the theory will be useful to us. Again, the preferred resource for the nonincreasing rearrangement and rearrangement-invariant Banach function norms is~\cite[Chapter 2]{BennetSharpley}. In the case of quasinorms, we turn to the two articles~\cite{MuNePeTu} and~\cite{Pesaquasi}. \par
We first introduce the concepts of the distribution function and of the nonincreasing rearrangement. Notice that the nonincreasing rearrangement is simply the generalized inverse of the distribution function.

\begin{defn}
    Let $f\in \M(R,\mu)$. We define the \textsl{distribution function} of $f$ as the function
    \begin{equation*}
        f_*(s):= \mu(\{t\in R: \abs{f(t)}>s\}), \ s\in[0,\infty),
    \end{equation*}
    and we define the \textsl{nonincreasing rearrangement} of $f$ as the function
    \begin{equation*}
        f^*(t):= \inf\{s\in [0,\infty): f_*(s)\leq t\}, \ t\in[0,\infty).
    \end{equation*}
\end{defn}

The basic properties of these functions can be found in~\cite[Chapter 2, Proposition 1.3]{BennetSharpley} and~\cite[Chapter 2, Proposition 1.7]{BennetSharpley}. We consider these properties well-known enough to omit them and use them without explicit reference. \par

Sometimes we work with resonant measure spaces, which are defined below. Note that this definition was originally shown as a characterization of a more complicated definition; see \cite[Chapter 2, Theorem 2.7]{BennetSharpley}.
\begin{defn}
    A $\sigma$-finite measure space $(R,\mu)$ is \textsl{resonant} if it is nonatomic or completely atomic with atoms of equal measure.
\end{defn}

Another crucial notion is equimeasurability.
\begin{defn}
    We say that two functions $f\in\M(R,\mu)$ and $g\in \M(S,\nu)$ are \textsl{equimeasurable} if $f_*=g_*$.
\end{defn}
For analyzing the local behavior of a function, we will make quite heavy use of a theorem that will allow us to explicitly construct the level set of prescribed measure for a given function by \say{reconstructing} the function from its rearrangement. More information, including proofs for the following statements, can be obtained in~\cite[Chapter 2.7]{BennetSharpley}.

\begin{defn}
    Let $(R,\mu)$ and $(S,\nu)$ be two measure spaces. A measurable mapping $\sigma\colon R\to S$ is said to be a \textsl{measure-preserving transformation} if for every $\nu$-measurable set $E\subset S$ and its inverse image $\sigma^{-1}(E)$ it holds that $\mu(\sigma^{-1}(E))=\nu(E)$.
\end{defn}

Naturally, this transformation preserves equimeasurability.

\begin{prop}\label{preserve_equimeasurability}
    Let $(R,\mu)$ and $(S,\nu)$ be two arbitrary $\sigma$-finite measure spaces, and let $\sigma\colon R\to S$ be a measure-preserving transformation. For $f\in\M_+(S,\nu)$, we define $g:=f\circ \sigma$. Then $g\in\M_+(R,\mu)$, and $f$ and $g$ are equimeasurable.
\end{prop}

Next up is the promised theorem due to J.~V.~Ryff, which originates from~\cite{Ryff}.

\begin{thm}\label{ryff}
    Let $(R,\mu)$ be a finite nonatomic measure space and let $f\in \M_+(R,\mu)$. Then there exists a measure preserving transformation $\sigma\colon R\to (0,\mu(R))$ such that $f=f^*\circ \sigma$ a.e.
\end{thm}

Notice that this theorem only works with a finite measure space. We can eliminate this assumption if we consider a function that decays at infinity.

\begin{cor}\label{ryff_cor}
    Let $(R,\mu)$ be an arbitrary infinite nonatomic measure space, and let $f\in \M_+(R,\mu)$. If 
    \begin{equation*}
        \lim_{t\to \infty} f^*(t)=0,
    \end{equation*}
    then there exists a measure preserving transformation $\sigma\colon \supp f \to \supp f^*$ such that $f=f^*\circ \sigma$ a.e.~on $\supp f$.
\end{cor}

We now define a subclass of (quasi-)Banach function spaces where the size of some function depends only on its distribution function. Again, a systematic study of the normed version of these spaces can be found in~\cite[Chapter 2.4]{BennetSharpley}. Generalizations of important results to quasinorms can be found in the recent paper~\cite{MuNePeTu}.

\begin{defn}
    Let $\norm{\cdot}_X$ be a quasi-Banach function norm and $X$ its corresponding quasi-Banach function space. If for every $f,g\in\M(R,\mu)$, $\norm{\cdot}_X$ satisfies
    \begin{enumerate}
        \item[(P6)] $\norm{f}_X=\norm{g}_X$ whenever $f_*=g_*$,
    \end{enumerate}
    we say that $\norm{\cdot}_X$ and $X$ are \textsl{rearrangement-invariant}, often abbreviated to r.i.
\end{defn}

The property (P6) has deep consequences, one of which is the following representation theorem. Essentially, we can replace an r.i.~quasinorm over any resonant measure space with a quasinorm over the real half line with the Lebesgue measure. The original result for norms is the classical Luxemburg representation theorem; see~\cite[Chapter 2, Theorem 4.10]{BennetSharpley}. The presented version for quasinorms was obtained in~\cite[Proposition 3.2]{pesaabs}, which is based on the construction from~\cite[Theorem 3.1]{MuNePeTu}. We will often work with the representation space without explicitly mentioning this theorem.

\begin{thm}
    Let $(R,\mu)$ be a resonant measure space and let $\norm{\cdot}_X$ be an r.i.~quasi-Banach function norm over $\M(R,\mu)$. Then there exists an r.i.~quasi-Banach function norm $\norm{\cdot}_{\overline{X}}$ on $\M([0,\infty), \lambda)$ such that for all $f\in\M(R,\mu)$, it holds $\norm{f}_X=\norm{f^*}_{\overline{X}}$. \par 
    Furthermore, $\norm{\cdot}_{\overline{X}}$ has the property (P5) whenever $\norm{\cdot}_X$ does and if $(R,\mu)$ is infinite and nonatomic, then $\norm{\cdot}_{\overline{X}}$ is uniquely determined.
\end{thm}

Quite recently, a new characterization of the absolute continuity of the norm in r.i.~spaces utilizing amalgam ideas has been proven in~\cite{pesaabs}. It allows us to narrow our focus down to merely two particular sequences of sets while showing the absolute continuity of the norm of a function. Note that in the case where the representation space $\overline{X}$ is not unique, this result and others are highly dependent on its specific construction. Hence, from this point on we always assume that the representation from~\cite[Definition 3.1]{pesaabs} is used.

\begin{thm}\label{ri_ac_char}
    Let $\norm{\cdot}_X$ be an r.i.~quasi-Banach function norm and $X$ its corresponding r.i.~quasi-Banach function space. Then $f\in X$ has an absolutely continuous norm if and only if
    \begin{equation*}
        \lim_{n\to \infty} \norm{f^*\chi_{[0,n^{-1})}}_{\overline{X}}=0,
    \end{equation*}
    and
    \begin{equation*}
        \lim_{n\to \infty} \norm{f^*\chi_{[n,\infty)}}_{\overline{X}}=0.
    \end{equation*}
\end{thm}

A quantitative version of Theorem~\ref{ri_ac_char} for Banach function spaces has also appeared earlier in~\cite{Kiwerski}. \par
Lastly, we introduce a quintessential example of r.i.~spaces, which is the class of Lorentz spaces. Notice that if $p=q$, we obtain the classical Lebesgue space $L^p(R,\mu)$ (see \cite[Chapter 2, Proposition 1.8]{BennetSharpley}).
\begin{defn}
    Let $p,q\in (0,\infty]$. For $f\in\M(R,\mu)$ we define the \textsl{Lorentz functional}
    \begin{equation*}
        \norm{f}_{L^{p,q}}:=\begin{dcases}
           \left( \int_0^\infty \left(t^{\frac{1}{p}} f^*(t)\right)^q \frac{dt}{t} \right)^{\frac{1}{q}}, &q\in (0,\infty), \\
           \sup_{t\in (0,\infty)} t^{\frac{1}{p}} f^*(t), &q=\infty,
        \end{dcases}
    \end{equation*}
    and, with it, the corresponding \textsl{Lorentz space} $L^{p,q}(R,\mu)$.
\end{defn}

\section{Weak Fatou property}\label{weak_fatou_sec}
In this section, we focus on developing tools to handle quasinorms that only satisfy the weak Fatou property. We do so mainly by slightly generalizing and adjusting the existing results presented in the book~\cite[Chapter 15, §65 and §66]{Integration}. To start with, we define the weak Fatou property. Note that if the functional has the property (P2), we can replace the limit superior with an ordinary limit.
\begin{defn}
    We say that a functional ${\norm{\cdot}_X}\colon \M_+(R,\mu)\rightarrow [0,\infty]$ has the \textsl{weak Fatou property} if for every $f_n,f\in \M_+(R,\mu)$  satisfying $f_n\nearrow f$ a.e.~and 
    \begin{equation*}
        \limsup_{n\to\infty} \norm{f_n}_X < \infty,
    \end{equation*}
    it holds that $\norm{f}_X<\infty$.
\end{defn}
The next theorem, originally due to I. Amemiya~\cite{Amemiya}, gives a characterization of the weak Fatou property on quasinormed spaces. We present a slightly modified version of the proof to fit into our context of quasinorms. 
\begin{thm}\label{weak_fatou_char}
    Let ${\norm{\cdot}_X}\colon \M_+(R,\mu)\to [0,\infty]$ be a functional that satisfies (Q1), (P2). Then $\norm{\cdot}_X$ has the weak Fatou property if and only if there exists $C\geq1$ such that for every $f_n,f\in \M_+(R,\mu)$ satisfying $f_n\nearrow f$ a.e.~and
    \begin{equation*}
        \lim_{n\to\infty} \norm{f_n}_X < \infty,
    \end{equation*}
    it holds that
    \begin{equation*}
        \norm{f}_X \leq C \lim_{n\to\infty} \norm{f_n}_X .
    \end{equation*}
\end{thm}
\begin{proof}
The sufficiency is clear. For necessity, suppose that no such $C\geq 1$ exists. Then there exist functions $f_{n,k}, f_k \in \M_+(R,\mu)$ such that $f_{n,k}\nearrow f_k$ a.e.~and
\begin{equation*}
    \norm{f_k}_X > k^3 C_X^k \lim_{n\to\infty} \norm{f_{n,k}}_X,
\end{equation*}
for every $k\in\N$. Note that it is impossible that $\lim_n \norm{f_{n,k}}_X=0$ for some $k\in\N$ because then by the monotonicity of $f_{n,k}$ and (Q1) of $\norm{\cdot}_X$, $f_{n,k}=0$ a.e., for all $n\in\N$, and so $f_{k}=0$ a.e.,~which makes the inequality not strict. Hence, multiplying by appropriate constants, we may assume that $\lim_n \norm{f_{n,k}}_X=k^{-2} C_X^{-k}$, and so $\norm{f_k}_X>k$, for all $k\in\N$. Now, for $n\in\N$, define
\begin{equation*}
    g_n:=\sum_{k=1}^n f_{n,k}.
\end{equation*}
Then $g_n$ is clearly an increasing sequence, and by Proposition~\ref{nekvindas_trick} we calculate that
\begin{equation*}
    \lim_{n\to\infty}\norm{g_n}_X \leq \lim_{n\to\infty}\sum_{k=1}^{n} C_X^{k} \norm{f_{n,k}}_X \leq \sum_{k=1}^{\infty} \frac{1}{k^2}<\infty.
\end{equation*}
Hence, the pointwise limit $g$ of $g_n$ satisfies $\norm{g}_X<\infty$ by the weak Fatou property. However,
\begin{equation*}
    g=\lim_{n\to\infty} g_n \geq \lim_{n\to\infty} f_{n,k}=f_k,
\end{equation*}
for any $k\in\N$, so by (P2) of $\norm{\cdot}_X$,
\begin{equation*}
    \norm{g}_X\geq\norm{f_k}_X>k,
\end{equation*}
for any $k\in\N$, which yields a contradiction.
\end{proof}

We will show that if a quasinorm on a space $X$ has the weak Fatou property, we can construct an appropriate equivalent quasinorm that has the property (P3). Hence, we may define spaces by quasinorms that only have the weak Fatou property and recover most of the available results for quasi-Banach function spaces using this equivalence. The idea of this construction was originally proposed by G.~G.~Lorentz in his unpublished work and was later published in the book~\cite[Chapter 15, §66]{Integration}.

\begin{defn}
    Let ${\norm{\cdot}_X}\colon \M_+(R,\mu)\to [0,\infty]$ be a functional that satisfies (Q1), (P2). We denote
    \begin{equation*}
        \norm{f}_{\fat(X)}:=\inf\left\{\lim_{n\to\infty} \norm{f_n}_X: f_n\nearrow f, \text{ } f_n\in\M_+(R,\mu) \text{ for every } n\in\N \right\}.
    \end{equation*}
\end{defn}

We will now show that this construction has all the desired properties.

\begin{thm}\label{fat_fatou}
    Let ${\norm{\cdot}_X}\colon \M_+(R,\mu)\to [0,\infty]$ be a functional that satisfies (Q1), (P2). Then $\norm{\cdot}_{\fat(X)}$ has the property (P3).
\end{thm}

We omit the proof of Theorem~\ref{fat_fatou} as it can be found in~\cite[Chapter 15, §66, Theorem 2]{Integration}. 

\begin{thm}\label{fat_quasinorm}
    Let ${\norm{\cdot}_X}\colon \M_+(R,\mu)\to [0,\infty]$ be a functional that satisfies (Q1), (P2), (P4). Then $\norm{\cdot}_{\fat(X)}$ is a quasi-Banach function norm that satisfies $\norm{f}_{\fat(X)}\leq\norm{f}_X$, for every $f\in \M_+(R,\mu).$
\end{thm}
\begin{proof}
    We only need to show (Q1) c) because (P3) is given by Theorem~\ref{fat_fatou} and the other properties are clear. Fix $f,g\in \M_+(R,\mu)$. Let $\varepsilon>0$ be arbitrary and find $f_n,g_n\in\M_+(R,\mu)$ such that $f_n\nearrow f$, $g_n\nearrow g$ a.e.,
    \begin{equation*}
        \norm{f_n}_{X} < \norm{f}_{\fat(X)} + \varepsilon,
    \end{equation*}
    and
    \begin{equation*}
        \norm{g_n}_{X} < \norm{g}_{\fat(X)} + \varepsilon.
    \end{equation*}
    So we have $f_n+g_n \nearrow f+g$ a.e.~and
    \begin{equation*}
        \norm{f+g}_{\fat(X)}\leq \norm{f_n+g_n}_X\leq C_X \left( \norm{f_n}_X + \norm{g_n}_X\right) \leq C_X \left( \norm{f}_{\fat(X)} + \norm{g}_{\fat(X)} + 2\varepsilon\right),
    \end{equation*}
    which concludes the proof.
\end{proof}

\begin{thm}\label{fat_equiv}
    Let ${\norm{\cdot}_X}\colon \M_+(R,\mu)\to [0,\infty]$ be a functional that satisfies (Q1), (P2). Then $\norm{\cdot}_X$ has the weak Fatou property if and only if 
    \begin{equation*}
        \norm{\cdot}_X \approx \norm{\cdot}_{\fat(X)},
    \end{equation*}
    on $\M(R,\mu)$.
\end{thm}
\begin{proof}
    The sufficiency is clear due to Theorem~\ref{fat_fatou}. For necessity, let $f\in\M_+(R,\mu)$ and $C\geq 1$ be the constant from Theorem~\ref{weak_fatou_char}. Then for every $f_n\in\M_+(R,\mu)$ such that $f_n\nearrow f$, we have 
    \begin{equation*}
        \norm{f}_X\leq C \lim_{n\to\infty} \norm{f_n}_X.
    \end{equation*}
    This immediately gives 
    \begin{equation*}
        \norm{f}_X\leq C \norm{f}_{\fat(X)}.
    \end{equation*}
    The other inequality is provided by Theorem~\ref{fat_quasinorm} and so the quasinorms are equivalent.
\end{proof}

To finish this section, we would like to present concrete examples of what the functional $\norm{\cdot}_{\fat(X)}$ may look like. The following theorem will help us with reaching that goal by showing that $\norm{\cdot}_{\fat(X)}$ coincides with the second associate norm if we start with a norm.

\begin{thm}\label{fat_second_dual}
    Let ${\norm{\cdot}_X}\colon \M_+(R,\mu)\to [0,\infty]$ be a functional that satisfies (P1), (P2), (P4). Then $\norm{\cdot}_{X''}=\norm{\cdot}_{\fat(X)}$. 
\end{thm}

The proof of Theorem~\ref{fat_second_dual} can be found, once again, in~\cite[Chapter 15, §71, Theorem 2]{Integration}. 

\begin{example}
    Consider the following highly nonstandard norm
    \begin{equation*}
        \norm{x_n}_{c_0}:=\norm{x_n}_{\ell^\infty} + \infty \cdot\limsup_{n\to\infty} \abs{x_n},
    \end{equation*}
    where we interpret $\infty\cdot 0=0$. This norm then generates the classical sequence space
    \begin{equation*}
        c_0:=\{(x_n)_{n=1}^\infty \subset \Co : \lim_{n\to \infty} x_n=0\}.
    \end{equation*}
    Note that we use this construction so that the space is defined by a norm and fits within our context. Then this functional does not have the weak Fatou property. It is well-known that the associate space of $c_0$ is the space of summable sequences $\ell^1$. It is also classical that the associate space to $\ell^1$ is the space of bounded sequences $\ell^\infty$. Hence, using Theorem~\ref{fat_second_dual}, we obtain $\norm{\cdot}_{\fat(c_0)}=\norm{\cdot}_{\ell^\infty}$. \par
    For another example, consider the sequence space given by the norm
    \begin{equation*}
        \norm{x_n}_X:=\norm{x_n}_{\ell^\infty} + \limsup_{n\to\infty} \abs{x_n}.
    \end{equation*}
    This norm is clearly equivalent to $\norm{\cdot}_{\ell^\infty}$ so it has the weak Fatou property, but it does not have the strong Fatou property. A counterexample can be easily constructed. By the same calculation as above, we obtain $\norm{\cdot}_{\fat(X)}=\norm{\cdot}_{\ell^\infty}$.
\end{example}

\section{Wiener--Luxemburg amalgam spaces}\label{amalgam_sec}

\subsection{Definition and basic properties}

This section deals with the new approach to Wiener--Luxemburg amalgam spaces. First, we recall the original definition. Throughout this section, we restrict ourselves to only nonatomic measure spaces with infinite measure, unless stated otherwise. In particular, the maximal fundamental function of every space is finite at every point due to Theorem~\ref{lenka4.4}. Such a restriction is completely natural in our context because we need an infinite measure space to speak about global behavior, and nonatomicity allows for sets of any measure. 

\begin{defn}\label{original_amalgam}
    Let $\norm{\cdot}_A$, $\norm{\cdot}_B$ be r.i.~quasi-Banach function norms and $A$, $B$ their corresponding r.i.~quasi-Banach function spaces. We define the \textsl{Wiener--Luxemburg amalgam norm} as 
    \begin{equation*}
        \norm{f}_{\WL(A,B)}:=\norm{f^*\chi_{[0,1]}}_{\overline{A}} + \norm{f^*\chi_{[1,\infty)}}_{\overline{B}},
    \end{equation*}
    for any $f\in\M(R,\mu)$.
\end{defn}

The idea of this space is to split the function and study its local and global behavior separately. Observe that to extract local behavior, we restrict its rearrangement to the interval $[0,1]$, and for the global behavior, we remove the interval $[0,1]$. So, locally, we want to find a set where the function is the largest, and globally, we want to remove this set. Here, the word \say{local} pertains to \say{on sets of finite measure}. In particular, these sets may still escape to infinity. Conversely, the word \say{global} refers to \say{outside of sets of finite measure}. We attempt to capture this idea in the following definition, which allows for general non-r.i.~quasi-Banach function norms. It will be useful for us to consider slightly more general operators. It is interesting to note that this definition also works for arbitrary measure spaces.

\begin{defn}\label{maximal_minimal_norm}
    Let $\norm{\cdot}_X$ be a quasi-Banach function norm and let $X$ be the corresponding quasi-Banach function space. We define the \textsl{maximal local norm} of $f\in X$
    \begin{equation*}
        \sigma_Xf(t):=\sup_{\substack{E\subset R \\ \mu(E)\leq t}} \norm{f \chi_{E}}_X, \ t>0,
    \end{equation*}
    and the \textsl{minimal global norm} of $f\in X$
    \begin{equation*}
        \delta_Xf(t):=\inf_{\substack{E\subset R \\ \mu(E)\leq t}} \norm{f \chi_{R\setminus E}}_X, \ t>0.
    \end{equation*}
    We also denote the evaluation at $1$ as
    \begin{equation*}
        \norm{f}_{X_{\loc}}:=\sigma_Xf(1),
    \end{equation*}
    and
    \begin{equation*}
        \norm{f}_{X_{\glob}}:=\delta_Xf(1).
    \end{equation*}
\end{defn}

\begin{rem}
    In the case of the space $L^1$, our notation coincides with the classical notation $L^1_{\loc}$ for the set of locally integrable functions. However, the condition in our definition is much stronger. Because we never need to work with the set of locally integrable functions in this text, the symbol $\norm{\cdot}_{L^1_{\loc}}$ will always denote the functional from Definition~\ref{maximal_minimal_norm}. 
\end{rem}

\begin{rem}
    In nonatomic measure spaces, it can be easily shown that
    \begin{equation*}
        \sigma_Xf(t)=\sup_{\substack{E\subset R \\\mu(E)= t}} \norm{f \chi_{E}}_X , \ t>0,
    \end{equation*}
    and 
    \begin{equation*}
        \delta_Xf(t)=\inf_{\substack{E\subset R \\\mu(E)= t}} \norm{f \chi_{R\setminus E}}_X, \ t>0.
    \end{equation*}
\end{rem}

Now, we can define the new Wiener--Luxemburg functional. Note that, as we will see later on, this functional is a quasinorm under reasonable assumptions on the local space. 

\begin{defn}\label{new_amalgam}
    Let $\norm{\cdot}_A$, $\norm{\cdot}_B$ be quasi-Banach function norms and $A$, $B$ their corresponding quasi-Banach function spaces. We define the \textsl{Wiener--Luxemburg amalgam norm} as
    \begin{equation*}
        \norm{f}_{\GWL(A,B)} := \norm{f}_{A_{\loc}} + \norm{f}_{B_{\glob}},
    \end{equation*}
    for any $f\in \M(R,\mu)$.
\end{defn}

Here, we use the notation $\GWL(A,B)$ to distinguish our new definition from the original definition using rearrangements. However, we still call it the Wiener--Luxemburg amalgam space because we will show later that this is in fact a direct extension. \par
An intuitive observation is that an amalgam of a space with itself coincides with the original space.

\begin{rem}\label{equivalent_quasinorm}
    It holds that $\norm{\cdot}_A\approx\norm{\cdot}_{\GWL(A,A)}$ for an arbitrary quasi-Banach function norm $\norm{\cdot}_A$. Indeed, for any $f\in\M(R,\mu)$, we have
    \begin{equation*}
        \norm{f}_{\GWL(A,A)} = \norm{f}_{A_{\loc}} + \norm{f}_{A_{\glob}}\leq 2 \norm{f}_A.
    \end{equation*}
    For the other inequality, let $f\in\M(R,\mu)$, $\varepsilon>0$ and let $E\subset R$, $\mu(E)\leq 1$ be such that
    \begin{equation*}
        \norm{f\chi_{R\setminus E}}_A < \norm{f}_{A_{\glob}} + \varepsilon.
    \end{equation*}
    Then
    \begin{equation*}
        \norm{f}_A \leq C_A\left(\norm{f\chi_E}_A + \norm{f\chi_{R\setminus E}}_A\right) < C_A\left(\norm{f}_{A_{\loc}}+\norm{f}_{A_{\glob}} + \varepsilon\right),
    \end{equation*}
    which shows the result.
\end{rem}

Now, we would like to study some properties of the defined operators and functionals. Namely, we prove two useful inequalities, and we show that both defined operators are finite at all points if they are finite in at least one point.

\begin{prop}\label{local_estimate}
    Let $\norm{\cdot}_X$ be a quasi-Banach function norm and let $X$ be the corresponding quasi-Banach function space. Then, for all $f\in \M(R,\mu)$, 
    \begin{equation*}
        \sigma_Xf(t_1) \leq \sigma_Xf(t_2) \leq \left(\sum_{n=1}^{\lceil t_2/t_1\rceil} C_X^{n}\right) \sigma_Xf(t_1),
    \end{equation*}
    for any $0<t_1<t_2$.
\end{prop}
\begin{proof}
    The first inequality can be easily shown by noticing that $\sigma_Xf$ is an increasing function. Let $\varepsilon>0$, $0<t_1<t_2$, and $f\in\M(R,\mu)$. From the definition of $\sigma_X$, we find a set $E\subset R$, $\mu(E)=t_2$ such that
    \begin{equation*}
        \sigma_Xf(t_2) < \norm{f \chi_{E}}_X + \varepsilon.
    \end{equation*}
    By the classical Sierpi\'nski theorem, we find $\lceil t_2/t_1\rceil$ pairwise disjoint sets $E_n \subset R$, $\mu(E_n)\leq t_1$ such that $E=\bigcup_n E_n$. Using Proposition~\ref{nekvindas_trick}, we calculate
    \begin{equation*}
        \sigma_Xf(t_2) < \norm{f \chi_{E}}_X + \varepsilon = \norm{\sum_{n=1}^{\lceil t_2/t_1\rceil} f\chi_{E_n}}_X + \varepsilon \leq \sum_{n=1}^{\lceil t_2/t_1\rceil} C_X^{n} \norm{f\chi_{E_n}}_X + \varepsilon
        \leq \sigma_Xf(t_1) \sum_{n=1}^{\lceil t_2/t_1\rceil} C_X^{n} + \varepsilon,
    \end{equation*}
    which proves the inequality.
\end{proof}

We would like to get a similar result for the minimal global norm. However, this kind of inequality is impossible to obtain in this case. A simple counterexample can be constructed by taking characteristic functions of two disjoint sets of a given measure. The following proposition solves this problem by leveraging an additional assumption on the local space. We note that this sort of problem will be a common theme throughout this text. Usually, in our proofs the local part will behave quite nicely, but working with the global part is going to be more complicated.

\begin{prop}\label{global_properties}
    Let $\norm{\cdot}_A$, $\norm{\cdot}_B$ be quasi-Banach function norms and $A$, $B$ their corresponding quasi-Banach function spaces, and let $\norm{\cdot}_A$ be admissible. Then, for all $f\in \M(R,\mu)$, the following statements are true:
    \begin{enumerate}[label=(\roman*), itemsep=0.5em]
        \item $\delta_Af(t)$ is either finite for all $t>0$ or infinite for all $t>0$,
        \item for all $0<t_1<t_2$, it holds
        \begin{equation*}
       \delta_Bf \left(t_2\right) \leq \delta_Bf \left(t_1\right) \leq C_B  \max\left\{1,\frac {\varphi^{\max}_B(t_2-t_1)}{\varphi^{\min}_A(t_1)}\right\} \left( \sigma_Af(t_2)+\delta_Bf(t_2)\right),
    \end{equation*}
    \end{enumerate}
    In particular,
    \begin{equation*}
        \delta_Bf \left(\frac{1}{2}\right) \leq C_B \max\left\{1,\frac{\varphi^{\max}_B(1/2)}{\varphi^{\min}_A(1/2)}\right\} \norm{f}_{\GWL(A,B)}.
    \end{equation*}
\end{prop}
\begin{proof}
    Let $f\in \M(R,\mu)$, $\varepsilon>0$, and $0<t_1<t_2$. Using the definition of $\delta_B$, we find $G\subset R$, $\mu(G)=t_2$ such that 
    \begin{equation*}
        \norm{f\chi_{R\setminus G}}_B < \delta_Bf(t_2) + \varepsilon.
    \end{equation*}
    By Theorem~\ref{ryff}, there exists a measure preserving transformation $\sigma\colon G \rightarrow (0,t_2)$ such that $(\abs{f})|_G=(f|_G)^* \circ \sigma$.
    We denote $\tilde{G}:=\sigma^{-1}((0,t_1])$. Then $\tilde{G}\subset G$, $\mu(\tilde{G})= t_1$, $\abs{f}\geq (f|_G)^*(t_1)$ in $\tilde{G}$ and
 $\abs{f}\leq (f|_G)^*(t_1)$ in $G\setminus \tilde{G}$. 
    We get
    \begin{equation}\label{inf_estimate}
        \begin{aligned}
            \delta_Bf  \left(t_1\right)
            &\leq  \norm{f\chi_{R\setminus \tilde{G}}}_B \\
            &\leq C_B\left(\norm{f\chi_{R\setminus G}}_B + \norm{f\chi_{G\setminus \tilde{G}}}_B\right) \\
            &\leq C_B\left(\delta_Bf(t_2) + \varepsilon + (f|_G)^*\left(t_1\right)\norm{\chi_{G\setminus \tilde{G}}}_B\right).
        \end{aligned}
    \end{equation}
    Now, we realize that the first inequality in the second statement is simply the monotonicity of the minimal global norm, and we notice that the terms other than the minimal global norm on the last line of~\eqref{inf_estimate} are all finite because
    \begin{equation*}
        \norm{\chi_{G\setminus \tilde{G}}}_B \leq \varphi_B^{\max}(t_2)<\infty.
    \end{equation*}
    Together, these facts yield the first statement. For the second inequality in the second statement, we use the admissibility of $\norm{\cdot}_A$ and the fact that $\mu(G\setminus \tilde{G})=t_2-t_1$ to continue the estimate~\eqref{inf_estimate}
    \begin{equation*}
        \begin{aligned}
            \delta_Bf \left(t_1\right)
            &\leq C_B\left(\delta_Bf(t_2) + \varepsilon + (f|_G)^*\left(t_1\right)\frac{\varphi^{\max}_B(t_2-t_1)}{\varphi^{\min}_A(t_1)}\norm{\chi_{\tilde{G}}}_A\right) \\
            &\leq C_B\left(\delta_Bf(t_2) + \varepsilon + \frac{\varphi^{\max}_B(t_2-t_1)}{\varphi^{\min}_A(t_1)}\norm{f\chi_{\tilde{G}}}_A\right) \\
            &\leq C_B\left(\delta_Bf(t_2) + \varepsilon + \frac{\varphi^{\max}_B(t_2-t_1)}{\varphi^{\min}_A(t_1)}\sigma_Af(t_2)\right).
        \end{aligned}
    \end{equation*}
    This proves the inequality.
\end{proof}

An immediate corollary of these two results is the fact that if we choose sets of any different measure in the definition of our amalgam functional, we obtain the same space. More precisely, the quasinorms are always equivalent. This will be used later in the study of continuous embeddings.

\begin{prop}\label{gwl_norm_equivalence}
    Let $\norm{\cdot}_A$, $\norm{\cdot}_B$ be quasi-Banach function norms and $A$, $B$ their corresponding quasi-Banach function spaces, and let $\norm{\cdot}_A$ be admissible. Then
    \begin{equation*}
        \sigma_A(\cdot)(t_1)+\delta_B(\cdot)(t_2)\approx \sigma_A(\cdot)(t_3)+\delta_B(\cdot)(t_4),
    \end{equation*}
    for all $t_1,t_2,t_3,t_4>0$.
\end{prop}
\begin{proof}
    The statement follows immediately from Proposition~\ref{local_estimate} and Proposition~\ref{global_properties}.
\end{proof}

Now, we can finally investigate whether our functional is a quasi-Banach function norm. As we shall see in the following lemma, the strong Fatou property is very easy to obtain for the local part.

\begin{lemma}\label{individual_fatou}
    Let $\norm{\cdot}_X$ be a quasi-Banach function norm and let $X$ be the corresponding quasi-Banach function space. Then $\norm{\cdot}_{X_{\loc}}$ has the strong Fatou property.
\end{lemma}
\begin{proof}
    Let $f_n,f\in \M_+(R,\mu)$ be  such that $f_n\nearrow f$ a.e. Using the property (P3) of $\norm{\cdot}_X$ and interchanging suprema, we obtain
    \begin{equation*}
        \lim_{n\to \infty} \norm{f_n}_{X_{\loc}}=\sup_{n\in \N} \sup_{\substack{E\subset R \\ \mu(E)\leq 1}} \norm{f_n\chi_E}_{X}=\sup_{\substack{E\subset R \\ \mu(E)\leq 1}} \sup_{n\in \N} \norm{f_n\chi_E}_{X} = \sup_{\substack{E\subset R \\ \mu(E)\leq 1}} \norm{f\chi_E}_{X} = \norm{f}_{X_{\loc}}.
    \end{equation*}
    From the property (P2) of $\norm{\cdot}_X$, it immediately follows that the convergence is monotone.
\end{proof}

Unfortunately, we are not able to show the strong Fatou property for the global norm. Notice that unlike in Lemma~\ref{individual_fatou}, the same calculation would result in an attempt to interchange a supremum and an infimum, which is rarely possible. So we resort to only proving the weak Fatou property and all other quasi-Banach function norm properties.

\begin{thm}\label{gwl_quasinorm}
    Let $\norm{\cdot}_A$, $\norm{\cdot}_B$ be quasi-Banach function norms and $A$, $B$ their corresponding quasi-Banach function spaces, and let $\norm{\cdot}_A$ be admissible. Then $\norm{\cdot}_{\GWL(A,B)}$ has the properties (Q1), (P2), (P4), and it has the weak Fatou property. Moreover, if $\norm{\cdot}_A$ has the property (P5), then $\norm{\cdot}_{\GWL(A,B)}$ has the property (P5).
\end{thm}
\begin{proof}
    Let $f,g\in \M_+(R,\mu)$. Properties (Q1) a), b), (P2), and (P4) are easy consequences of the properties of $\norm{\cdot}_A$ and $\norm{\cdot}_B$. \par
    First, we show (Q1) c). 
    The local estimate 
    \begin{equation*}
        \norm{f+g}_{A_{\loc}} \leq C_A \left( \norm{f}_{A_{\loc}} + \norm{g}_{A_{\loc}}\right),
    \end{equation*}
    is easily obtained from the subadditivity of the supremum. For the global estimate, 
    let $\varepsilon>0$. From the definition of $\delta_B$, we find sets $E_1$ and $E_2$ such that $\mu(E_1)=\mu(E_2)=1/2$,
    \begin{equation*}
        \norm{f\chi_{R\setminus E_1}}_B < \delta_Bf\left(\frac{1}{2}\right) + \varepsilon,
    \end{equation*}
    and
    \begin{equation*}
        \norm{g\chi_{R\setminus E_2}}_B < \delta_B g\left(\frac{1}{2}\right) + \varepsilon.
    \end{equation*}
    We put $E_0:=E_1\cup E_2$. Then $\mu(E_0)\leq 1$ and using Proposition~\ref{global_properties}, we obtain
    \begin{equation*}
        \begin{aligned}
            \norm{f+g}_{B_{\glob}}&\leq \norm{(f+g)\chi_{R\setminus E_0}}_B \\
            &\leq C_B \left(\norm{f\chi_{R\setminus E_1}}_B+\norm{g\chi_{R\setminus E_2}}_B\right) \\
            &< 2C_B \varepsilon + C_B\left( \delta_Bf\left(\frac{1}{2}\right) + \delta_Bg\left(\frac{1}{2}\right)\right) \\
            &\leq 2C_B \varepsilon + (C_B)^2 \max\left\{1,\frac{\varphi^{\max}_B(1/2)}{\varphi^{\min}_A(1/2)}\right\} \left( \norm{f}_{\GWL(A,B)} + \norm{g}_{\GWL(A,B)}\right).
        \end{aligned}
    \end{equation*}\par
    Now, we show that $\norm{\cdot}_{\GWL(A,B)}$ has the weak Fatou property. Let $f_n,f\in \M_+(R,\mu)$ be such that $f_n\nearrow f$, and 
    \begin{equation*}
        \lim_{n\to \infty} \norm{f_n}_{\GWL(A,B)} < \infty.
    \end{equation*}
    By the definition of $\norm{\cdot}_{B_{\glob}}$, for every $n\in\N$, we find a set $E_n\subset R$, $\mu(E_n)=1$ such that 
    \begin{equation*}
        \norm{f_n \chi_{R\setminus E_n}}_{B} < \norm{f_n}_{B_{\glob}} + 1.
    \end{equation*}
    Notice that $\norm{f_n}_{B_{\glob}}\leq\lim_n \norm{f_n}_{\GWL(A,B)}< \infty$ due to the monotonicity. For every $n\in \N$, we use Proposition~\ref{ryff} to find a measure preserving transformation $\sigma_n\colon E_n \rightarrow (0,1)$ such that $f_n|_{E_n}=(f_n|_{E_n})^*\circ \sigma_n$. We denote $\tilde{E}_n:=\sigma_n^{-1}((0,2^{-n}])$. Then $\mu(\tilde{E}_n)=2^{-n}$, and $f_n\leq (f_n|_{E_n})^*(2^{-n})$ on $E_n\setminus \tilde{E}_n$. Let
    \begin{equation*}
        E:=\bigcup_{n=1}^\infty \tilde{E}_n.
    \end{equation*}
    Then 
    \begin{equation*}
        \mu(E)\leq \sum_{n=1}^\infty 2^{-n}=1.
    \end{equation*}
    Now, using the property (P3) of $\norm{\cdot}_B$, we find $n\in \N$ large enough so that
    \begin{equation*}
        \norm{f\chi_{R\setminus E}}_B <\norm{f_n\chi_{R\setminus E}}_B + 1.
    \end{equation*}
    To finish, we estimate
    \begin{equation*}
        \begin{aligned}
            \norm{f}_{B_{\glob}} &\leq \norm{f\chi_{R\setminus E}}_B 
            < \norm{f_n\chi_{R\setminus E}}_B +1 \\
            &\leq C_B\left(\norm{f_n\chi_{R\setminus E_n}}_B+\norm{f_n \chi_{E_n \setminus E}}_B\right) + 1 \\
            &\leq C_B\left(\norm{f_n\chi_{R\setminus E_n}}_B+\norm{f_n \chi_{E_n \setminus \tilde{E}_n}}_B\right) + 1 \\
            &\leq C_B\left(\norm{f_n\chi_{R\setminus E_n}}_B+ (f_n|_{E_n})^*(2^{-n}) \norm{\chi_{E_n \setminus \tilde{E}_n}}_B\right) + 1 \\
            &\leq C_B\left(\norm{f_n\chi_{R\setminus E_n}}_B+ (f_n|_{E_n})^*(2^{-n}) \varphi_B^{\max}(1)\right) + 1 < \infty.
         \end{aligned}
    \end{equation*}
    This and the fact that $\norm{f}_{A_{\loc}}<\infty$, due to Lemma~\ref{individual_fatou}, gives that $\norm{f}_{\GWL(A,B)}<\infty$.
    \par
    Finally, suppose that $\norm{\cdot}_A$ has the property (P5) and 
    let $E\subset R$ be such that $\mu(E)<\infty$. Denote $n:=\ceil{\mu(E)}$. Using the classical Sierpi\'nski theorem, 
    we find $n$ sets $E_k$ such that $\mu(E_k)\leq1$ for every $k\in\N ,k\leq n$, and $\bigcup_k E_k=E$. 
    Then using the property (P5) of $\norm{\cdot}_A$, we get 
    \begin{equation*}
        \begin{aligned}
            \int_E f \, d\mu&= \sum_{k=1}^n \int_{E_k} f \, d\mu \\
            &\leq \sum_{k=1}^n C_{E_k} \norm{f\chi_{E_k}}_A \\ 
            &\leq n \max_{k\leq n} C_{E_k} \norm{f}_{A_{\loc}} \\
         &\leq n \max_{k\leq n} C_{E_k} \norm{f}_{\GWL(A,B)}, 
        \end{aligned}
    \end{equation*}
    where $C_{E_k}$ denotes the constant from the property (P5) for the set $E_k$. This shows that $\norm{\cdot}_{\GWL(A,B)}$ has the property (P5).
\end{proof}

Here, the tools we discussed in the previous section come into play. We can use them to prove that our space $\GWL(A,B)$ is, in fact, a quasi-Banach function space. Establishing this is a crucial fact; without it, our construction would be of little use. Furthermore, this classification allows us to leverage many of the results already available for these spaces.

\begin{cor}\label{gwl_quasinormability}
     Let $\norm{\cdot}_A$, $\norm{\cdot}_B$ be quasi-Banach function norms and $A$, $B$ their corresponding quasi-Banach function spaces, and let $\norm{\cdot}_A$ be admissible. Then $\GWL(A,B)$ is a quasi-Banach function space endowed with the quasi-Banach function norm $\norm{\cdot}_{\fat(\GWL(A,B))}$, and 
     \begin{equation*}
         \norm{\cdot}_{\fat(\GWL(A,B))}\approx\norm{\cdot}_{\GWL(A,B)}.
     \end{equation*}
\end{cor}
\begin{proof}
    This follows immediately from Theorems~\ref{fat_quasinorm}, ~ \ref{fat_equiv},~and ~ \ref{gwl_quasinorm}.
\end{proof}

\subsection{Consistency}
The next question that we investigate is whether our new definition is consistent with the original definition of Wiener--Luxemburg amalgams. That is, if we go back to considering r.i.~spaces, we would like the spaces defined in Definition~\ref{original_amalgam} and~\ref{new_amalgam} to coincide. We will again work with the local and global parts separately. In the case of the local part, the answer turns out to be fairly straightforward, as shown in the following proposition. More general versions of these results are presented, which allow for a resonant measure space. However, they prove to be relevant outside the scope of this paper. We refer the reader to the related article~\cite{pesa2026amalgamapproachcompactnessquasibanach}, where these results are also featured and used. Complete proofs are still included, as these results are still quite new and very relevant to our context and methods.

\begin{prop}\label{local_ri_equiv}
    Let $(R, \mu)$ be resonant, let $\norm{\cdot}_X$ be an r.i.~quasi-Banach function norm, and let $X$ be the corresponding r.i.~quasi-Banach function space. Let $t\geq0$ be in the range of $\mu$. 
    Then, for all $f\in \M_0(R,\mu)$, we have
    \begin{equation*}
        \sigma_X f(t)=\norm{f^*\chi_{[0,t)}}_{\overline{X}}.
    \end{equation*}
\end{prop}

\begin{proof}
    Let $f\in \M_0(R,\mu)$. For the first inequality, let $\varepsilon>0$ and by the definition of $\sigma_X f$ find a set $G\subset \mathcal{R}$, $\mu(G)=t$ such that 
    \begin{equation*}
        \sigma_X f(t) < \norm{f\chi_G}_X + \varepsilon.
    \end{equation*}
    Then using the properties of the nonincreasing rearrangement and the fact that $f\chi_G$ and $(f\chi_G)^*$ are equimeasurable, with the latter function being supported inside $[0, t)$, we obtain
    \begin{equation*}
        \sigma_X f(t) < \norm{f\chi_G}_X + \varepsilon = \norm{(f\chi_G)^* \chi_{[0,t)}}_{\overline{X}} + \varepsilon \leq \norm{f^* \chi_{[0,t)}}_{\overline{X}} + \varepsilon,
    \end{equation*}
    which gives
    \begin{equation*}
        \sigma_X f(t)\leq\norm{f^*\chi_{[0,t)}}_{\overline{X}}.
    \end{equation*}
    
    For the second inequality, we start by only considering $f\in \M_0(\mathcal{R} ,\mu)$ nonnegative and with support of finite measure. We denote $E:=\supp f$. Note that the statement is trivial if $\mu(E)\leq t$, so we may assume that $\mu(E)>t$. We now need to distinguish two cases:
    \begin{itemize}[leftmargin=5mm]
        \item If the measure space $(R,\mu)$ is nonatomic, then by Theorem~\ref{ryff} there exists a measure preserving transformation $\sigma\colon E \rightarrow (0,\mu(E))$ such that $(f)|_E=(f|_E)^*\circ \sigma$ a.e. Next, we put $H:=\sigma^{-1}((0,t))$. Then $\mu(H)=t$~and $((f|_E)^*\chi_{[0,t)}) \circ \sigma = (f)|_E \chi_H$ a.e. Hence, $f^*\chi_{[0,t)}$ and $f \chi_H$ are equimeasurable.
        \item If $(R,\mu)$ is completely atomic with atoms of measure equal to $\alpha$, then $f^*$ is a right-continuous decreasing step function with jumps occurring only at points $k\alpha$, $k\in\N$. For every $k\in\N$ such that $k\alpha\leq t$ we find a distinct $x_k\in R$ satisfying $f(x_k)=f^*(k\alpha)$; this is possible because we assume that $\mu(E) < \infty$. Let $H\subset R$ be the collection of all $x_k$ such that $k\alpha\leq t$. From this construction, it is clear that the functions $f^*\chi_{[0,t)}$ and $f\chi_H$ are equimeasurable.
    \end{itemize}
    In both cases, this yields
    \begin{equation*}
        \norm{f^*\chi_{[0,t)}}_{\overline{X}} = \norm{f \chi_H}_X \leq \sigma_X f(t).
    \end{equation*}
    For a general nonnegative $f\in\M(R,\mu)$, we let $f_n:=f\chi_{R_n}$ for every $n\in\N$, where $R_n\nearrow R$ is the sequence from the $\sigma$-finiteness of $(R,\mu)$. The previously proven case gives
    \begin{equation*}
        \sigma_X f_n(t)=\norm{f^*_n\chi_{[0,t)}}_{\overline{X}},
    \end{equation*}
    for every $n\in\N$. 
    
    We may now take the limit on both sides. Indeed, on the right-hand side, the convergence follows by using the properties of the nonincreasing rearrangement and the property (P3) of $\norm{\cdot}_{\overline{X}}$, while on the left-hand side one only has to interchange suprema and apply the property (P3) of $\norm{\cdot}_{X}$. Finally, the extension from nonnegative functions to general ones is trivial.
\end{proof}

The answer is not as simple for the global part, as we were only able to obtain the following result with a substantially more involved proof.

\begin{prop}\label{global_norm_rearrangement}
   Let $(R, \mu)$ be resonant, let $\norm{\cdot}_X$ be an r.i.~quasi-Banach function norm, and let $X$ be the corresponding r.i.~quasi-Banach function space. Let $t\geq0$ be in the range of $\mu$. Then there exist constants $C_1,C_2>0$ such that it holds for all $f\in \M_0(R,\mu)$ that
    \begin{equation*}
       C_1 \delta_Xf(t) \leq \norm{f^*\chi_{[t,\infty)}}_{\overline{X}} \leq  C_2 \delta_Xf\left(\frac{t}{2}\right).
    \end{equation*}
\end{prop}

\begin{proof}
    We start with the second inequality. Let $\varepsilon>0$, and $f\in\M_0(R,\mu)$. By the definition of $\delta_X$, we find a set $H\subset R$, $\mu(H) \leq t/2$ such that
    \begin{equation*}
        \norm{f\chi_{R\setminus H}}_X < \delta_Xf\left(\frac{t}{2}\right) + \varepsilon.
    \end{equation*}
    From the properties of the nonincreasing rearrangement, we get the following pointwise inequality
    \begin{equation*}
        g^*(s)-h^*\left(\frac{s}{2}\right)\leq (g-h)^*\left(\frac{s}{2}\right),
    \end{equation*}
    for all $g,h\in\M_0(R,\mu)$ and $s \in (0,\infty)$. 
    We notice that $(f\chi_H)^*\left(s/2\right)=0$ for all $s\in [t,\infty)$,
    so for all $s\in [t,\infty)$, we get
    \begin{equation*}
        f^*(s)=f^*(s)-(f\chi_H)^*\left(\frac{s}{2}\right)\leq(f\chi_{R\setminus H})^*\left(\frac{s}{2}\right).
    \end{equation*}
    Using this inequality and the boundedness of the dilation operator $D_\frac{1}{2}$ (see~\cite[Theorem 3.23]{Pesaquasi}), we calculate
    \begin{equation*}
        \begin{aligned}
            \norm{f^*\chi_{[t,\infty)}}_{\overline{X}}
                        &\leq \norm{ D_{\frac{1}{2}} \left(\left(f\chi_{R\setminus H}\right)^* \right)\chi_{[t,\infty)}}_{\overline{X}} \\
                        &\leq \norm{D_{\frac{1}{2}}}_{\overline{X}\rightarrow \overline{X}} \norm{\left(f\chi_{R\setminus H}\right)^*}_{\overline{X}} \\
                        &=\norm{D_{\frac{1}{2}}}_{\overline{X}\rightarrow \overline{X}} \norm{f\chi_{R\setminus H}}_{X} \\
                        &< \norm{D_{\frac{1}{2}}}_{\overline{X}\rightarrow \overline{X}} \left( \delta_Xf\left(\frac{t}{2}\right) + \varepsilon\right).
        \end{aligned}
    \end{equation*}
    
    We now move to the first inequality. Without loss of generality, we may consider only functions $f\in \M_+(R,\mu)\cap \M_0(R,\mu)$ and we denote
    \begin{equation*}
        \alpha:=\lim_{t\to\infty} f^*(t).
    \end{equation*}
    If $\mu(\supp f)\leq t$, the statement is clear because both terms evaluate to 0. Hence, we may only consider $\mu(\supp f)>t$. 
    
    First, suppose that $\alpha=0$. We need to distinguish two cases:
    \begin{itemize}[leftmargin=5mm]
        \item If the measure space $(R,\mu)$ is nonatomic, then by Theorem~\ref{ryff}, there exists a measure preserving transformation $\sigma\colon \supp f\rightarrow \supp f^*$ such that $f=f^*\circ \sigma$ a.e.~on $\supp f$. We put $G:=\sigma^{-1}((0,t))$. Then $\mu(G)=t$ and $f\chi_{R\setminus G}=f^*\chi_{[t,\infty)}\circ \sigma$ a.e.~on $\supp f$. And so, $f^*\chi_{[t,\infty)}$ and $f\chi_{R\setminus G}$ are equimeasurable.
        \item If $(R,\mu)$ is completely atomic with atoms of measure equal to $\beta$, then $f^*$ is a right-continuous decreasing step function with jumps occurring only at points $k\beta$, $k\in\N$. Since $\alpha=0$, it is easy to construct an ordering $\sigma\colon \N \to \R$ such that $f(\sigma(k))=f^*((k-1)\beta)$, $k\in\N$. Let $G=\sigma(\{1,\ldots,t/\beta\})$. Then $\mu(G)=t$ and it is clear from the construction that $f^*\chi_{[t,\infty)}$ and $f\chi_{R\setminus G}$ are equimeasurable.
    \end{itemize}
    In both cases, it follows that
    \begin{equation*}
        \delta_Xf(t) \leq \norm{f\chi_{R\setminus G}}_X=\norm{f^*\chi_{[t,\infty)}}_{\overline{X}}.
    \end{equation*}
    
    Next, suppose that $\alpha>0$ and that $\norm{f^*\chi_{[t,\infty)}}_{\overline{X}}<\infty$, otherwise the inequality is trivial. Since we have $\alpha=\lim_{s\to\infty }f^*(s)\chi_{[t,\infty)}(s)$, the result~\cite[Theorem 4.16]{MuNePeTu} states that there exists a constant $C>0$ such that
    \begin{equation*}
        \norm{\chi_{[0,\infty)}}_{\overline{X}} = \norm{\chi_{[t,\infty)}}_{\overline{X}} = \norm{\chi_{R \setminus E}}_X = \norm{\chi_R}_X = C < \infty,
    \end{equation*}
    for every $t>0$ and every $E \subset R$, $\mu(E) < \infty$ (the first three equalities are due to equimeasurability of the functions in question). Let now $f_0:=\max\{f,\alpha\}$. Then $f^*=f_0^*$ and $f\leq f_0$. Using these facts, the calculation in the previous case for the function $f_0-\alpha$, and the properties of the nonincreasing rearrangement, we obtain
    \begin{equation*}
        \begin{aligned}
            \delta_Xf(t)&\leq\inf_{\substack{E\subset R \\\mu(E) \leq t}} \norm{f_0 \chi_{R\setminus E}}_X \\
            &\leq C_X\inf_{\substack{E\subset R \\\mu(E) \leq t}} \left( \norm{\left(f_0-\alpha\right) \chi_{R\setminus E}}_X + \norm{\alpha \chi_{R\setminus E}}_X\right) \\
            &= C_X\left(\delta_X \left(f_0-\alpha\right)(t) + \alpha C \right)\\
            &\leq C_X\left(\norm{\left(f_0-\alpha\right)^*\chi_{[t,\infty)}}_{\overline{X}} + \alpha \norm{\chi_{[t,\infty)}}_{\overline{X}} \right)\\
            &\leq C_X\left(\norm{f^*\chi_{[t,\infty)}}_{\overline{X}} + \norm{f^*\chi_{[t,\infty)}}_{\overline{X}} \right)\\
            &= 2C_X \norm{f^*\chi_{[t,\infty)}}_{\overline{X}}.
        \end{aligned}
    \end{equation*}
\end{proof}

Finally, we can combine these results to obtain the desired conclusion. That is, for r.i.~quasi-Banach function norms the two amalgam quasinorms are, indeed, equivalent and thus produce the same topology. Notice that rearrangement-invariant spaces are always admissible, so we do not need to explicitly assume it.

\begin{thm}\label{gwl=wl}
    Let $\norm{\cdot}_A$, $\norm{\cdot}_B$ be r.i.~quasi-Banach function norms and $A$, $B$ their corresponding r.i.~quasi-Banach function spaces. Then 
    \begin{equation*}
        \norm{\cdot}_{\GWL(A,B)}\approx \norm{\cdot}_{\WL(A,B)}.
    \end{equation*}
\end{thm}
\begin{proof}
    We simply use Propositions~\ref{gwl_norm_equivalence},~\ref{local_ri_equiv}, and~\ref{global_norm_rearrangement}.
\end{proof}

\subsection{Associate space and normability}

Now, we would like to characterize the associate space of our amalgam space. It turns out that the fundamental result from~\cite{Amalgam} still holds. However, our new construction of these spaces requires for a more sophisticated approach. The main new ingredient that we need is an observation about admissible spaces. That is, the functions inside such space cannot grow infinitely. We now define this property. Once again, note that sets of finite measure can still \say{escape to infinity}.

\begin{defn}
    Let $t>0$. We say that a function $f\in\M(R,\mu)$ is \textsl{globally bounded} if there exists a set $E\subset R$, $\mu(E)<\infty$ satisfying $\norm{f\chi_{R\setminus E}}_{L^\infty} < \infty$.
\end{defn}

With this we can state our observation. Recall that we consider the measure space $(R,\mu)$ to be nonatomic.

\begin{lemma}\label{cut_to_bound}
    Let $\norm{\cdot}_X$ be an admissible quasi-Banach function norm and let $X$ be the corresponding quasi-Banach function space. Let $f\in \M(R,\mu)$ be such that $\norm{f}_{X_{\loc}}<\infty$. Then $f$ is globally bounded.
\end{lemma}
\begin{proof}
    Suppose that a function $f\in\M(R,\mu)$ such that $\norm{f}_{X_{\loc}}<\infty$ is not globally bounded. Then ${\mu\left(\{x\in R : \abs{f(x)} >n\}\right)} = \infty$, for every $n\in \N$. Indeed, if this is false, then there exists $n\in\N$ such that $f$ is globally bounded with $E=\{x\in R : \abs{f(x)} >n\}$. Hence, we may use the classical Sierpi\'nski theorem to find sets $G_n\subset R$ such that $\mu(G_n)=1$, and $\abs{f}>n$ on $G_n$. Due to the admissibility of $X$, this gives
    \begin{equation*}
        \norm{f}_{X_{\loc}}\geq \norm{f\chi_{G_n}}_X\geq n \varphi_X^{\min}(1),
    \end{equation*}
    for every $n\in \N$. This yields a contradiction with our assumptions.
\end{proof}

This example shows that the conclusion of the previous lemma may not hold if the space is not admissible.

\begin{example}
 Let $w(x):=\min\{1,1/x^3\}$ and consider the quasi-Banach function space $L^1(w)([0,\infty),\lambda)$. In Example~\ref{non_admissibility}, we have shown that this space is not admissible. Let $f(x):=x$, $x>0$. Then 
    \begin{equation*}
        \norm{f}_{L^1(w)}=\int_0^\infty \min\left\{x,\frac{1}{x^2}\right\} \, dx < \infty.
    \end{equation*}
    However, $f$ is obviously not globally bounded because it has a blow-up at infinity.
\end{example}

We finally move on to our characterization of the associate space. From now on the symbol $\GWL(A,B)$ denotes the space induced by $\norm{\cdot}_{\GWL(A,B)}$, which can be considered a quasi-Banach function space due to Theorem~\ref{gwl_quasinormability}.

\begin{thm}\label{duality}
 Let $\norm{\cdot}_A$, $\norm{\cdot}_B$ be quasi-Banach function norms with the property (P5) and $A$, $B$ their corresponding quasi-Banach function spaces. Then, there exists a constant $C>0$ such that 
\begin{equation}\label{dual_equivalence}
    2^{-1} \norm{f}_{(\GWL(A,B))'} \leq \norm{f}_{\GWL(A',B')} \leq C \norm{f}_{(\GWL(A,B))'},
\end{equation}
for every $f\in\M(R,\mu)$. \par
In other words
\begin{equation*}
    (\GWL(A,B))'= \GWL(A',B').
\end{equation*}
\end{thm}
\begin{proof}
Theorem~\ref{second_dual} states that  both $\norm{\cdot}_{A'}$ and $\norm{\cdot}_{B'}$ are Banach function norms. In particular, they are admissible by Proposition~\ref{p5admissible}. We start with the left inequality in~\eqref{dual_equivalence}. Let $f\in\M(R,\mu)$. Note that we may use the open unit ball instead of the closed unit ball in the definition of the associate norm. Let $g\in \GWL(A,B)$, $\norm{g}_{\GWL(A,B)}<1$, and $\varepsilon>0$. Then 
\begin{equation}\label{sup<1}
    \norm{g}_{A_{\loc}}<1,
\end{equation}
and
\begin{equation}\label{inf<1}
    \norm{g}_{B_{\glob}}<1.
\end{equation}
By~\eqref{inf<1}, there exists $E_1\subset R$, $\mu(E_1)=1$ such that
\begin{equation*}
    \norm{g\chi_{R\setminus E_1}}_B < 1.
\end{equation*}
Next, we find $E_2\subset R$, $\mu(E_2)=1$ such that
\begin{equation}\label{inf_approx}
    \sup_{\norm{h}_B \leq 1} \int_{R\setminus E_2} \abs{fh} \, d\mu < \inf_{\mu(E)=1} \sup_{\norm{h}_B \leq 1} \int_{R\setminus E} \abs{fh} \, d\mu + \varepsilon.
\end{equation}
Using the definition of the associate norm, \eqref{sup<1} and \eqref{inf_approx}, we calculate
\begin{equation*}
    \begin{aligned}
        \int_R \abs{fg} \, d\mu &= \int_{E_1} \abs{fg} \, d\mu +\int_{R\setminus E_1} \abs{fg} \, d\mu \\
                                &\leq \int_{E_1} \abs{fg} \, d\mu +\int_{(R\setminus E_1)\setminus E_2} \abs{fg} \, d\mu
                                                                + \int_{(R\setminus E_1)\cap E_2} \abs{fg} \, d\mu \\
                                &\leq \int_{E_1} \abs{fg} \, d\mu +\int_{R\setminus E_2} \abs{fg} \chi_{R\setminus E_1} \, d\mu
                                                                + \int_{E_2} \abs{fg} \, d\mu \\
                                &\leq 2 \sup_{\mu(E)=1} \sup_{\norm{h}_A\leq 1} \int_{E} \abs{fh} \, d\mu + \sup_{\norm{h}_B\leq 1} \int_{R\setminus E_2} \abs{fh} \, d\mu \\
                                &< 2 \sup_{\mu(E)=1} \sup_{\norm{h}_A\leq 1} \int_{E} \abs{fh} \, d\mu + \inf_{\mu(E)=1} \sup_{\norm{h}_B \leq 1} \int_{R\setminus E} \abs{fh} \, d\mu + \varepsilon \\ 
                                &\leq 2 \norm{f}_{\GWL(A',B')} + \varepsilon.
    \end{aligned}
\end{equation*}
The other inequality in~\eqref{dual_equivalence} is going to require more work. We will show it indirectly using Theorem~\ref{continuous_embeddings}. That is, the inequality holds if and only if $(\GWL(A,B))'\subset \GWL(A',B')$. Assume that $f\notin \GWL(A',B')$. Then at least one of the following statements is true
\begin{equation}\label{local_infinity}
    \sup_{\substack{E\subset R \\ \mu(E)=1}} \norm{f \chi_E}_{A'} = \infty, 
\end{equation}
or
\begin{equation}\label{global_infinity}
    \inf_{\substack{E\subset R \\ \mu(E)=1}} \norm{f \chi_{R\setminus E}}_{B'} = \infty.
\end{equation}
If~\eqref{local_infinity} holds, by the definition of the associate norm there exist $g_n\in A$, $\norm{g}_A\leq 1$ and $E_n\subset R$, $\mu(E_n)=1$ such that 
\begin{equation*}
    \int_{E_n} \abs{fg_n} \, d\mu \nearrow \infty.
\end{equation*}
We observe that $\norm{g_n \chi_{E_n}}_{\GWL(A,B)}\leq 1$, so
\begin{equation}
    \norm{f}_{(\GWL(A,B))'}= \sup_{\norm{h}_{\GWL(A,B)} \leq 1} \int_R \abs{fh} \, d\mu \geq \int_{E_n} \abs{fg_n} \, d\mu \nearrow \infty.
\end{equation}
If~\eqref{global_infinity} is true and~\eqref{local_infinity} is false, then by Proposition~\ref{local_estimate}, we have $\sigma_{A'}f(t)<\infty$, for any $t\geq 0$, and by Proposition~\ref{global_properties}, $\delta_{B'}f(t)=\infty$, for any $t>0$. Thus, for any $E\subset R$, $\mu(E)<\infty$, by Theorem~\ref{associate_char}, there exists $g_E$ such that $\norm{g_E}_B\leq1$ and 
\begin{equation}\label{infinite_integral}
    \int_{R\setminus E} \abs{fg_E} \, d\mu=\infty.
\end{equation}
Because $\norm{f}_{A'_{\loc}}<\infty$, we use Lemma~\ref{cut_to_bound} to find $E_1\subset R$, $\mu(E_1)<\infty$ and $C_1>0$ such that $\norm{f\chi_{R\setminus {E_1}}}_{L^\infty} \leq C_1$. Let $g\in \M(R,\mu)$, $\norm{g}_B \leq 1$ be such that~\eqref{infinite_integral} holds for $E_1$. Because $B$ is admissible and $\norm{g}_{B_{\loc}}\leq \norm{g}_B\leq1$, we again use Lemma~\ref{cut_to_bound} to find $E_2\subset R$, $\mu(E_2)<\infty$ and $C_2>0$ such that $\norm{g\chi_{R\setminus {E_2}}}_{L^\infty} \leq C_2$. Using the property (P5) of $\norm{\cdot}_B$, we calculate that 
\begin{equation*}
    \int_{(R\setminus E_1)\cap E_2} \abs{fg} \, d\mu \leq C_1 \int_{E_2} \abs{g} \, d\mu \leq C_1 C_{E_2} \norm{g}_B < \infty. 
\end{equation*}
Hence, by~\eqref{infinite_integral} it must be that 
\begin{equation}\label{infinite_integral2}
    \int_{(R\setminus E_1)\cap (R\setminus E_2)} \abs{fg} \, d\mu = \infty.
\end{equation}
We define $h:= g\chi_{(R\setminus E_1)\cap (R\setminus E_2)}$. Then $\norm{h}_{B_{\glob}}< \infty$ and
\begin{equation*}
    \norm{h}_{A_{\loc}} \leq \norm{g \chi_{R\setminus E_2}}_{A_{\loc}} \leq C_2 \varphi_A^{\max}(1)< \infty. 
\end{equation*}
So we have $h\in \GWL(A,B)$ and by~\eqref{infinite_integral2}
\begin{equation*}
    \int_R \abs{fh} \, d\mu = \infty, 
\end{equation*}
which shows that $f\notin (\GWL(A,B))'$.
\end{proof}
An important corollary of this theorem is the normability of our amalgam space if our original spaces are normable. 
\begin{cor}\label{normability}
     Let $\norm{\cdot}_A$, $\norm{\cdot}_B$ be Banach function norms and $A$, $B$ their corresponding Banach function spaces. Then there exists a Banach function norm $\norm{\cdot}_X$ such that 
     \begin{equation*}
         \norm{\cdot}_{\GWL(A,B)}\approx \norm{\cdot}_X.
     \end{equation*}
In other words, $\GWL(A,B)$ is a Banach function space endowed with the norm $\norm{\cdot}_X$.
\end{cor}
\begin{proof}
Due to Theorem~\ref{second_dual}, we have $\norm{\cdot}_A=\norm{\cdot}_{A''}$ and $\norm{\cdot}_B=\norm{\cdot}_{B''}$.  Using this and Theorem~\ref{duality}, we obtain
\begin{equation*}
    \norm{\cdot}_{\GWL(A,B)}=\norm{\cdot}_{\GWL(A'',B'')}\approx \norm{\cdot}_{(\GWL(A',B'))'}.
\end{equation*}
Theorems~\ref{second_dual}, and~\ref{gwl_quasinorm} show that $\norm{\cdot}_{(\GWL(A',B'))'}$ is a Banach function norm, which concludes the proof.
\end{proof}

\subsection{Embeddings}

We continue by showing that the already known embedding results for Wiener--Luxemburg amalgam spaces translate nicely to our new context. As a corollary of this, we obtain an extension of a well-known classical result.

\begin{thm}\label{amalgam_embeddings}
    Let $\norm{\cdot}_A$, $\norm{\cdot}_B$, $\norm{\cdot}_C$, $\norm{\cdot}_D$ be quasi-Banach function norms and $A$, $B$, $C$, $D$ their corresponding quasi-Banach function spaces, and let $\norm{\cdot}_A,\norm{\cdot}_C$ be admissible. Then the following statements are true:
    \begin{enumerate}[label=(\roman*)]
        \item It holds that $\GWL(A,B)\hookrightarrow \GWL(C,B)$ if and only if 
        \begin{equation}\label{local_stronger}
            \norm{f}_{A_{\loc}} < \infty \implies \norm{f}_{C_{\loc}} < \infty,
        \end{equation}
        for every $f\in \M(R,\mu)$.
        \item It holds that $\GWL(A,B)\hookrightarrow \GWL(A,D)$ if and only if 
        \begin{equation}\label{global_stronger}
            \norm{f}_{B_{\glob}} < \infty \implies \norm{f}_{D_{\glob}} < \infty,
        \end{equation}
        for every globally bounded $f\in \M(R,\mu)$.
        \item It holds that $\GWL(A,B)\hookrightarrow \GWL(C,D)$ if and only if~\eqref{local_stronger} and~\eqref{global_stronger} are true for every $f\in \M(R,\mu)$ and every globally bounded $f\in \M(R,\mu)$, respectively.
    \end{enumerate}
\end{thm}
\begin{proof}
    In the first two cases, the sufficiency is simple because it follows directly from Theorem~\ref{continuous_embeddings} and Corollary~\ref{gwl_quasinormability}. In the second case, one only needs to realize that all functions in $\GWL(A,B)$ are globally bounded because of Lemma~\ref{cut_to_bound}. Sufficiency in the third case simply follows from the first two because
    \begin{equation*}
        \GWL(A,B)\hookrightarrow\GWL(A,D)\hookrightarrow\GWL(C,D).
    \end{equation*} \par
    To show the necessity in the first case, let $f\in \M(R,\mu)$ be such that
    \begin{equation*}
        \norm{f}_{A_{\loc}} < \infty \text{, but } \norm{f}_{C_{\loc}} = \infty.
    \end{equation*}
    For every $n\in\N$, we may find a set $E_n\subset R$, $\mu(E)=1$ such that 
    \begin{equation*}
        \norm{f\chi_{E_n}}_C>n\norm{f}_{A_{\loc}}.
    \end{equation*}
    For every $n\in\N$, let $f_n:=f\chi_{E_n}$. It is easy to see that $\norm{f_n}_{B_{\glob}}=0$. Hence,
    \begin{equation*}
        \norm{f_n}_{\GWL(C,B)}=\norm{f_n}_{C_{\loc}}>n\norm{f}_{A_{\loc}}=n\left(\norm{f}_{A_{\loc}}+\norm{f_n}_{B_{\glob}}\right)=n\norm{f_n}_{\GWL(A,B)}, 
    \end{equation*}
    and so $\GWL(A,B)\not\hookrightarrow\GWL(C,B)$. \par
    For necessity in the second case, again let $f\in \M(R,\mu)$ be globally bounded and satisfy 
    \begin{equation*}
        \norm{f}_{B_{\glob}} < \infty \text{, but } \norm{f}_{D_{\glob}} = \infty.
    \end{equation*}
    Let $E\subset R$, $\mu(E)<\infty$ be such that $\norm{f\chi_{R\setminus E}}_{L^\infty}<\infty$. Denote $f_0:=f\chi_{R\setminus E}$. Then we have
    \begin{equation*}
        \norm{f_0}_{A_{\loc}} \leq \norm{f\chi_{R\setminus E}}_{L^\infty} \varphi_A^{\max}(1)<\infty,
    \end{equation*}
    and because $\abs{f_0}\leq\abs{f}$,
    \begin{equation*}
        \norm{f_0}_{B_{\glob}}\leq\norm{f}_{B_{\glob}}<\infty.
    \end{equation*}
    It remains to show that
    \begin{equation*}
        \norm{f_0}_{D_{\glob}} = \infty.
    \end{equation*}
    By Proposition~\ref{global_properties}, we know that $\delta_Df(\mu(E)+1)=\infty$. Now, realize that if we take a set $H\subset R$, $\mu(H)\leq1$ and consider the set $(R\setminus E)\setminus H=R\setminus (E\cup H)$, we are removing a set of measure at most $\mu(E)+1$ from $R$; hence, 
    \begin{equation*}
        \infty = \delta_Df(\mu(E)+1)= \inf_{\substack{F\subset R \\ \mu(F)=\mu(E)+1}} \norm{f\chi_{R \setminus F}}_D\leq \inf_{\substack{H\subset R \\ \mu(H)=1}} \norm{f_0\chi_{R\setminus H}}_D=\norm{f_0}_{D_{\glob}},
    \end{equation*}
    which shows that $\GWL(A,B)\not\hookrightarrow \GWL(A,D)$. \par
    To show the necessity for the third case, we use the same procedures as in the previous cases. That is, if~\eqref{local_stronger} is false, we use the construction from $(i)$, and if~\eqref{global_stronger} is false, we use the construction from $(ii)$.
\end{proof}

\begin{defn}
    Let $\norm{\cdot}_A$, $\norm{\cdot}_B$, $\norm{\cdot}_C$, $\norm{\cdot}_D$ be quasi-Banach function norms and $A$, $B$, $C$, $D$ their corresponding quasi-Banach function spaces, and let $\norm{\cdot}_A,\norm{\cdot}_C$ be admissible. If the condition~\eqref{local_stronger} holds, we say that the local component of $\norm{\cdot}_A$ is stronger than the local component of $\norm{\cdot}_C$. Similarly, if~\eqref{global_stronger} holds, we say that the global component of $\norm{\cdot}_B$ is stronger than the global component of $\norm{\cdot}_D$.
\end{defn}

As remarked in the original paper~\cite{Amalgam}, this approach of looking at embeddings locally and globally may be used to compare Lorentz (and Lebesgue) spaces over infinite measure spaces

\begin{example}\label{lorentz_comparison}
    Let $p_1,p_2,q_1,q_2\in (0,\infty]$ such that $p_1\not = p_2$ and let $\norm{\cdot}_{L^{p_1,q_1}}$ and $\norm{\cdot}_{L^{p_2,q_2}}$ be the corresponding Lorentz functionals. Then
    \begin{enumerate}[label=(\roman*)]
        \item the local component of $\norm{\cdot}_{L^{p_1,q_1}}$ is stronger than the local component of $\norm{\cdot}_{L^{p_2,q_2}}$ if and only if $p_1>p_2$,
        \item the global component of $\norm{\cdot}_{L^{p_1,q_1}}$ is stronger than the global component of $\norm{\cdot}_{L^{p_2,q_2}}$ if and only if $p_1<p_2$.
    \end{enumerate}
\end{example}

The following theorem shows that the amalgam space lies somewhere in between the sum and intersection of spaces. Under certain conditions, it may also align with one of these extremes. The idea of this theorem is directly taken from~\cite{pesa2025sumsrearrangementinvariantquasibanachfunction} where it was proven for the original Wiener--Luxemburg amalgam spaces.

\begin{thm}\label{sums_intersections}
    Let $\norm{\cdot}_A$, $\norm{\cdot}_B$ be quasi-Banach function norms and $A$, $B$ their corresponding quasi-Banach function spaces, and let $\norm{\cdot}_A$ be admissible. Then
    \begin{equation*}
        A\cap B \hookrightarrow \GWL(A,B) \hookrightarrow A+B.
    \end{equation*}
    Moreover, if $\norm{\cdot}_B$ is admissible, we assume that the local component of $\norm{\cdot}_A$ is stronger than the local component of $\norm{\cdot}_B$, and the global component of $\norm{\cdot}_B$ is stronger than the global component of $\norm{\cdot}_A$, then
    \begin{equation*}
        \begin{aligned}
            \norm{\cdot}_{\GWL(A,B)}&\approx \norm{\cdot}_{A\cap B}, \\
            \norm{\cdot}_{\GWL(B,A)}&\approx \norm{\cdot}_{A+ B}.
        \end{aligned}
    \end{equation*}
\end{thm}
\begin{proof}
    For the first set of embeddings, let $f\in A\cap B$. Then
    \begin{equation*}
        \norm{f}_{\GWL(A,B)}=\norm{f}_{A_{\loc}} + \norm{f}_{B_{\glob}} \leq \norm{f}_A+ \norm{f}_B\leq 2 \norm{f}_{A\cap B}.
    \end{equation*}
    Now, let $f\in \GWL(A,B)$, $\varepsilon>0$, and let $E\subset R$, $\mu(E)=1$ be such that
    \begin{equation*}
        \norm{f\chi_{R\setminus E}}_B < \norm{f}_{B_{\glob}} + \varepsilon.
    \end{equation*}
 From this we have $f\chi_{R\setminus E}\in B$. We may also observe that $f\chi_E\in A$ because $f\in\GWL(A,B)$. Hence, 
 \begin{equation*}
     \norm{f}_{A+B}\leq \norm{f\chi_E}_A + \norm{f\chi_{R\setminus E}}_B < \norm{f}_{A_{\loc}} + \norm{f}_{B_{\glob}} + \varepsilon = \norm{f}_{\GWL(A,B)} + \varepsilon,
 \end{equation*}
 which yields the other embedding. \par
 To prove the second statement, it remains to show
 \begin{equation*}
     \begin{aligned}
         \GWL(A,B) &\hookrightarrow A \cap B, \\
         A+B &\hookrightarrow \GWL(B,A).
     \end{aligned}
 \end{equation*}
 Our assumptions combined with Remark~\ref{equivalent_quasinorm} and Theorem~\ref{amalgam_embeddings} immediately give 
 \begin{equation*}
     \begin{aligned}
         \GWL(A,B) &\hookrightarrow A, \\
         \GWL(A,B) &\hookrightarrow B,
     \end{aligned}
 \end{equation*}
 which produces the first embedding. On the other hand, we also obtain
 \begin{equation*}
     \begin{aligned}
         A &\hookrightarrow \GWL(B,A), \\
         B &\hookrightarrow \GWL(B,A).
     \end{aligned}
 \end{equation*}
 Hence, there exists a constant $C>0$ such that for every $f=f_A+f_B$, $f_A\in A$, $f_B\in B$ it holds
 \begin{equation*}
     \begin{aligned}
         \norm{f}_{\GWL(B,A)} &\leq C_{\GWL(B,A)} \left(\norm{f_A}_{\GWL(B,A)} + \norm{f_B}_{\GWL(B,A)}\right) \\&\leq C C_{\GWL(B,A)} \left(\norm{f_A}_{A} + \norm{f_B}_{B}\right).
     \end{aligned}
 \end{equation*}
 The second embedding now follows from taking the infimum over such decompositions on the right-hand side.
\end{proof}

In the next theorem, we study extremal local and global embeddings into $L^1$ and $L^\infty$. 

\begin{thm}\label{extremal_embeddings}
    Let $\norm{\cdot}_A$, $\norm{\cdot}_B$ be quasi-Banach function norms and $A$, $B$ their corresponding quasi-Banach function spaces, and let $\norm{\cdot}_A$ be admissible. Then
    \begin{enumerate}[label=(\roman*)]
        \item $\GWL(L^\infty,B) \hookrightarrow \GWL(A,B)$,
        \item $\GWL(A,B) \hookrightarrow \GWL(A,L^\infty)$.
        \item If $\norm{\cdot}_A$ has the property (P5), then $\GWL(A,B) \hookrightarrow \GWL(L^1,B)$.
        \item If $\norm{\cdot}_B$ is a Banach function norm, then $\GWL(A,L^1) \hookrightarrow \GWL(A,B)$.
    \end{enumerate}
\end{thm}

\begin{proof}
    Let $f\in \M(R,\mu)$. We obtain the first statement from the inequality
    \begin{equation*}
        \norm{f\chi_{E}}_{A}\leq \norm{f\chi_E}_{L^\infty} \norm{\chi_E}_A \leq \norm{f}_{L^\infty_{\loc}} \varphi_A^{\max}(1),
    \end{equation*}
    for any $E\subset R$, $\mu(E)\leq 1$. \par
    
    The second statement follows from the fact that the condition~\eqref{global_stronger} in Theorem~\ref{amalgam_embeddings} with $\norm{\cdot}_D=\norm{\cdot}_{L^\infty}$ is met automatically for any globally bounded function due to Proposition~\ref{global_properties}. \par
    
    Due to the property (P5) of $\norm{\cdot}_A$, Theorem~\ref{second_dual} gives that $\norm{\cdot}_{A'}$ is a Banach function norm. Hence, using  Theorem~\ref{holder} we obtain the inequality
    \begin{equation*}
        \int_E \abs{f} \, d\mu \leq \norm{\chi_E}_{A'} \norm{f\chi_E}_A \leq \varphi_{A'}^{\max}(1) \norm{f}_{A_{\loc}},
    \end{equation*}
    for any $E\subset R$, $\mu(E)\leq 1$, which proves the third statement. \par
    
      Finally, for the fourth statement we observe that by Remark~\ref{equivalent_quasinorm} and the second statement we have $B'\hookrightarrow \GWL(B',L^\infty)$. From this it easily follows that $(\GWL(B',L^\infty))'\hookrightarrow B''$. Thus, Theorems~\ref{second_dual} and ~ \ref{duality} yield
    \begin{equation*}
        \GWL(B,L^1)=\left(\GWL(B',L^\infty)\right)'\hookrightarrow B''=B.
    \end{equation*}
    Hence, Theorem~\ref{amalgam_embeddings} asserts that~\eqref{global_stronger} is true for $L^1$ and $B$, which again by Theorem~\ref{amalgam_embeddings} shows the desired embedding.
\end{proof}

\begin{rem}\label{linfty_l1_combinations}
    From Theorems~\ref{amalgam_embeddings} and~\ref{extremal_embeddings} it easily follows that the local component of $L^\infty$ is stronger than the local component of $L^1$ and that the global component of $L^1$ is stronger than the global component of $L^\infty$. Hence, Theorem~\ref{sums_intersections} shows that
        \begin{equation*}
        \begin{aligned}
            \norm{\cdot}_{\GWL(L^\infty,L^1)}&\approx \norm{\cdot}_{L^1\cap L^\infty}, \\
            \norm{\cdot}_{\GWL(L^1,L^\infty)}&\approx \norm{\cdot}_{L^1+ L^\infty}.
        \end{aligned}
    \end{equation*}
\end{rem}

We may use these results to substantially generalize the classical result that all r.i.~Banach function spaces lie in between $L^1\cap L^\infty$ and $L^1+L^\infty$, see~\cite[Chapter 2, Theorem 6.6]{BennetSharpley}. In particular, we completely remove the need for rearrangement-invariance.

\begin{cor}\label{space_sandwich}
    Let $\norm{\cdot}_X$ be a quasi-Banach function norm with the property (P5) and $X$ its corresponding quasi-Banach function space. Then 
    \begin{equation*}
        X \hookrightarrow L^1 + L^\infty.
    \end{equation*}
    Furthermore, if $\norm{\cdot}_X$ is a Banach function norm, then
    \begin{equation*}
        L^1 \cap L^\infty \hookrightarrow X.
    \end{equation*}
\end{cor}

\begin{proof}
    By Proposition~\ref{p5admissible}, $\norm{\cdot}_X$ is admissible. Due to Remark~\ref{linfty_l1_combinations}, we have $\GWL(L^1,L^\infty)=L^1+L^\infty$ and $\GWL(L^\infty,L^1)=L^1\cap L^\infty$. The result now follows directly from Theorem~\ref{extremal_embeddings} and Remark~\ref{equivalent_quasinorm}.
\end{proof}

We would like to investigate the optimality of the assumptions in Corollary~\ref{space_sandwich}. It turns out that the conditions for the $L^1+L^\infty$ embedding are indeed sharp.

\begin{prop}
    Let $\norm{\cdot}_X$ be a quasi-Banach function norm and $X$ its corresponding quasi-Banach function space. Then $\norm{\cdot}_X$ has the property (P5) if and only if
    \begin{equation*}
        X\hookrightarrow L^1 + L^\infty.
    \end{equation*}
\end{prop}

\begin{proof}
    The necessity is given by Corollary~\ref{space_sandwich}. For the sufficiency, let $f\in \M_+(R,\mu)$ and $E\subset R$, $\mu(E)<\infty$. By our assumption, there exists a constant $C>0$ independent of $f$ such that
    \begin{equation}\label{l1linfty_embed}
        \int_0^1 f^* \, d\lambda =\norm{f}_{L^1+L^\infty}\leq C \norm{f}_X,
    \end{equation}
    where the first equality can be found in~\cite[Chapter 2, Theorem 6.2]{BennetSharpley}. First, we consider $\mu(E)\leq1$. Then the Hardy--Littlewood inequality (see~\cite[Chapter 2, Theorem 2.2]{BennetSharpley}) combined with~\eqref{l1linfty_embed} immediately give
    \begin{equation*}
        \int_E f \, d\mu \leq \int_0^{\mu(E)} f^* \, d\lambda \leq \int_0^1 f^* \, d\lambda \leq C \norm{f}_X.
    \end{equation*}
    Next, if $\mu(E)>1$, we use the fact that the integral mean of a decreasing function is a decreasing function. That is,
    \begin{equation*}
        \frac{1}{\mu(E)}\int_0^{\mu(E)} f^* \, d\lambda \leq \int_0^1 f^* \, d\lambda.
    \end{equation*}
    We again combine this with the Hardy--Littlewood inequality and~\eqref{l1linfty_embed} to obtain
    \begin{equation*}
        \int_E f \, d\mu \leq \int_0^{\mu(E)} f^* \, d\lambda \leq \mu(E) \int_0^1 f^* \, d\lambda \leq \mu(E) C \norm{f}_X.
    \end{equation*}
\end{proof}

In the case of the $L^1\cap L^\infty$ embedding, the condition is not sharp, as we show in the following example.
\begin{example}
    We will utilize amalgams to find an example. First, consider the nonnormable space $L^{2,\frac{1}{2}}$. Then by Proposition~\ref{amalgam_embeddings} and Example~\ref{lorentz_comparison} we have
    \begin{equation*}
        L^1\cap L^\infty=\GWL(L^\infty,L^1) \hookrightarrow \GWL(L^{2,\frac{1}{2}}, L^{2,\frac{1}{2}})=L^{2,\frac{1}{2}}.
    \end{equation*}
    So, the embedding may hold even if our space is not a Banach function space.
\end{example}

\subsection{Absolute continuity of the norm}

In the final section, we turn our attention to absolute continuity of the norm in Wiener--Luxemburg amalgam spaces. The main tool that we will use to prove our characterization in Theorem~\ref{amalgam_ac} is the following generalization of Theorem~\ref{ri_ac_char} to non-r.i.~spaces. This theorem is also featured in the closely related article~\cite{pesa2026amalgamapproachcompactnessquasibanach}.

\begin{thm}\label{ac_char}
Let $\norm{\cdot}_X$ be a quasi-Banach function norm and let $X$ be the corresponding quasi-Banach function space. Let $f\in X$. Then $f\in X_a$ if and only if the following conditions hold:
\begin{equation}\label{local}
    \lim_{n\rightarrow\infty} \sigma_Xf\left(n^{-1}\right) = 0,
\end{equation}
and
\begin{equation}\label{global}
    \lim_{n\rightarrow\infty} \delta_Xf(n) = 0.
\end{equation}
\end{thm}
\begin{proof}
We start with the necessity for~\eqref{local}. Let $\varepsilon>0$. 
For every $n\in\N$, we find a set $E_n\subset R$ such that $\mu(E_n)\leq\frac{1}{n}$ and 
\begin{equation*}
    \sigma_Xf\left(n^{-1}\right)  -\varepsilon < \norm{f \chi_{E_n}}_X.
\end{equation*}
We have $\mu(E_n)\rightarrow 0$. Hence, there exists a subsequence $E_{n_{k}}$ such that $\chi_{E_{n_{k}}}\rightarrow 0$ a.e. Because $f\in X_a$, we get that there exists $k_0\in\N$ such that for every $k\in\N, k\geq k_0$,
\begin{equation*}
    \norm{f \chi_{E_{n_{k}}}}_X < \varepsilon.
\end{equation*}
Thus, for every $k\in\N, k\geq k_0$, we get
\begin{equation*}
    \sigma_Xf\left(n_k^{-1}\right)< \norm{f \chi_{E_{n_k}}}_X + \varepsilon < 2 \varepsilon,
\end{equation*}
which is enough due to the monotonicity of our sequence. \par 
For the necessity of~\eqref{global}, we use the $\sigma$-finiteness of $(R,\mu)$ to find a sequence of sets $E_n\subset R$ such that $E_n \nearrow R$ and $\mu(E_n)<n$, for every $n\in\N$. Then $\chi_{R\setminus E_n}\rightarrow 0$ a.e. Let $\varepsilon >0$. Because $f\in X_a$, we find $n_0\in\N$ such that for every $n\in\N, n\geq n_0$ we have
\begin{equation*}
    \norm{f \chi_{R\setminus E_n}}_X<\varepsilon.
\end{equation*}
We obtain that for every $n\in\N, n\geq n_0$
\begin{equation*}
    \delta_Xf(n) \leq \norm{f \chi_{R\setminus E_n}}_X < \varepsilon.
\end{equation*} 
\par
For sufficiency, let $\varepsilon>0$ and let $\chi_{E_n}\rightarrow 0$ a.e.~for some $E_n\subset R$. By~\eqref{global}, 
we find $n_0\in\N$ such that 
\begin{equation*}
    \delta_Xf(n_0) < \varepsilon.
\end{equation*}
Hence, there exists a set $E_0$ such that $\mu(E_0)<n_0$ and 
\begin{equation}\label{E0}
     \norm{f \chi_{R\setminus E_0}}_X<\varepsilon.
\end{equation}
We denote $\tilde{E}_n = E_n \cap E_0$ and note that $\mu(\tilde{E}_n) \rightarrow 0$ by the dominated convergence theorem for the Lebesgue integral. Next, by~\eqref{local}, we find $n_1\in\N$  
such that for every set $E\subset R$, $\mu(E)< \frac{1}{n_1}$, we have 
\begin{equation*}
     \norm{f \chi_{E}}_X<\varepsilon.
\end{equation*}
Because $\mu(\tilde{E}_n) \rightarrow 0$, we find $n_2\in\N$ such that for every $n\in\N, n\geq n_2$ we have $\mu(\tilde{E}_n) < \frac{1}{n_1}$ and by 
the previous equation
\begin{equation}\label{EK}
    \norm{f \chi_{\tilde{E}_n}}_X<\varepsilon,
\end{equation}
for every $n\in \N$, $n\geq n_2$. Using~\eqref{E0} and~\eqref{EK}, we obtain that for every $n\in\N, n\geq n_2$, 
\begin{equation*}
    \norm{f \chi_{E_n}}_X \leq C_X\left(\norm{f\chi_{\tilde{E}_k}}_X+\norm{f\chi_{R\setminus E_0}}_X\right) < 2 C_X \varepsilon.
\end{equation*}
\end{proof}

The next result that we need is the following statement. Essentially, if a function decays globally in an admissible space $X$, it has to decay globally uniformly pointwise.

\begin{prop}\label{linfinity_delta}
    Let $\norm{\cdot}_X$ be an admissible quasi-Banach function norm and $X$ its corresponding quasi-Banach function space. Let $f\in\M(R,\mu)$. If
    \begin{equation*}
        \lim_{n\to \infty}\delta_X f(n)=0,
    \end{equation*}
    then
    \begin{equation*}
        \lim_{n\to \infty}\delta_{L^\infty} f(n)=0.
    \end{equation*}
\end{prop}

\begin{proof}
    By Remark~\ref{equivalent_quasinorm} and Theorem~\ref{extremal_embeddings} there exists a constant $C>0$ such that 
    \begin{equation*}
        \norm{f}_{\GWL(X,L^\infty)}\leq C\norm{f}_X.
    \end{equation*}
    From this it follows that for every $n\in \N$,
    \begin{equation*}
        \delta_{\GWL(X,L^\infty)}f (n) \leq C \delta_Xf(n).
    \end{equation*}
    Now, we calculate that for every $n\in\N$,
    \begin{equation*}
        \begin{aligned}
            \delta_{\GWL(X,L^\infty)}f (n)&= \inf_{\substack{E\subset R \\ \mu(E)=n}}\left(\norm{f\chi_{R\setminus E}}_{X_{\loc}} + \norm{f\chi_{R\setminus E}}_{L^\infty_{\glob}}\right) \\
            &\geq \inf_{\substack{E\subset R \\ \mu(E)=n}}\norm{f\chi_{R\setminus E}}_{L^\infty_{\glob}} \\
            &= \inf_{\substack{E\subset R \\ \mu(E)=n}} \inf_{\substack{G\subset R \\ \mu(G)=1}} \norm{f\chi_{R\setminus E}\chi_{R\setminus G}}_{L^\infty} \\
            &\geq \inf_{\substack{E\subset R \\ \mu(E)\leq n+1}} \norm{f\chi_{R\setminus E}}_{L^\infty} \\
            &=\delta_{L^\infty} f (n+1).
        \end{aligned}
    \end{equation*}
    The conclusion now follows from these inequalities.
\end{proof}

We now present a natural characterization of absolute continuity of the norm in Wiener--Luxemburg amalgam spaces. Compared to Theorem~\ref{ac_char}, we only need to investigate the local absolute continuity of the norm in the local space $A$ and the global absolute continuity of the norm in the global space $B$.

\begin{thm}\label{amalgam_ac}
    Let $\norm{\cdot}_A$, $\norm{\cdot}_B$ be admissible quasi-Banach function norms and $A$, $B$ their corresponding quasi-Banach function spaces. Let $f\in\GWL(A,B)$. Then $f\in\GWL(A,B)_a$ if and only if
    \begin{equation}\label{sigma_limit}
        \lim_{n\rightarrow\infty} \sigma_Af\left(n^{-1}\right) = 0,
    \end{equation}
    and
    \begin{equation}\label{delta_limit}
        \lim_{n\to \infty}\delta_B f(n)=0.
    \end{equation}
\end{thm}

\begin{proof}
    By Theorem~\ref{ac_char} we need to show that~\eqref{sigma_limit} and~\eqref{delta_limit} hold if and only if both
    \begin{equation*}
        \lim_{n\rightarrow\infty} \sigma_{\GWL(A,B)}f\left(n^{-1}\right) = 0,
    \end{equation*}
    and
    \begin{equation*}
        \lim_{n\to \infty}\delta_{\GWL(A,B)} f(n)=0.
    \end{equation*}
    For the maximal local norm, we observe that for every $n\in\N$,
    \begin{equation*}
        \sigma_{\GWL(A,B)}f\left(n^{-1}\right)= \sigma_Af\left(n^{-1}\right).
    \end{equation*}
    So it only remains to show the equivalence for the minimal global norm. For one inequality, we repeat essentially the same calculation from the previous proof. That is, for every $n\in \N$,
    \begin{equation*}
        \begin{aligned}
            \delta_{\GWL(A,B)}f (n)&= \inf_{\substack{E\subset R \\ \mu(E)=n}}\left(\norm{f\chi_{R\setminus E}}_{A_{\loc}} + \norm{f\chi_{R\setminus E}}_{B_{\glob}}\right) \\
            &\geq \inf_{\substack{E\subset R \\ \mu(E)=n}}\norm{f\chi_{R\setminus E}}_{B_{\glob}} \\
            &= \inf_{\substack{E\subset R \\ \mu(E)=n}} \inf_{\substack{G\subset R \\ \mu(G)=1}} \norm{f\chi_{R\setminus E}\chi_{R\setminus G}}_{B} \\
            &\geq \inf_{\substack{E\subset R \\ \mu(E)\leq n+1}} \norm{f\chi_{R\setminus E}}_{B} \\
            &=\delta_{B} f (n+1).
        \end{aligned}
    \end{equation*}
    This shows the first implication. \par
    Now, if~\eqref{delta_limit} holds, Proposition~\ref{linfinity_delta} gives that
    \begin{equation*}
        \lim_{n\to \infty}\delta_{L^\infty} f(n)=0.
    \end{equation*}
    From here we may find a sequence of sets $E_n\subset R$, $\mu(E_n)=m(n)$, where $m\colon \N \to \N$ is a strictly increasing function, such that $\lim_n m(n)=\infty$ and
    \begin{equation*}
        \lim_{n\to \infty} \norm{f\chi_{R\setminus E_n}}_{L^\infty} = 0.
    \end{equation*}
    Now, we calculate that for every $n\in\N$,
    \begin{equation*}
        \begin{aligned}
            \delta_{\GWL(A,B)}f (n+m(n))&= \inf_{\substack{E\subset R \\ \mu(E)\leq n+m(n)}} \left(\norm{f\chi_{R\setminus E}}_{A_{\loc}} + \norm{f\chi_{R\setminus E}}_{B_{\glob}}\right) \\
            &\leq \inf_{\substack{E\subset R \\ \mu(E)=n}} \left(\norm{f\chi_{R\setminus (E\cup E_n)}}_{A_{\loc}} + \norm{f\chi_{R\setminus (E\cup E_n)}}_{B_{\glob}}\right) \\
            &\leq \norm{f\chi_{R\setminus E_n}}_{A_{\loc}} + \inf_{\substack{E\subset R \\ \mu(E)=n}} \norm{f\chi_{R\setminus E}}_{B_{\glob}} \\
            &\leq \norm{f\chi_{R\setminus E_n}}_{L^\infty} \varphi_A^{\max}(1) + \delta_B f (n).
        \end{aligned}
    \end{equation*}
    The right-hand side goes to zero, so 
    \begin{equation*}
        \lim_{n\to\infty}\delta_{\GWL(A,B)}f (n+m(n)) = 0.
    \end{equation*}
    This yields the other implication because $\delta_{\GWL(A,B)}f$ is a decreasing function.
\end{proof}

We may use our previous results to transfer this characterization back to the r.i.~context of the original Wiener--Luxemburg amalgam spaces and Theorem~\ref{ri_ac_char}. Again, note that r.i.~spaces are always admissible.

\begin{cor}
     Let $\norm{\cdot}_A$, $\norm{\cdot}_B$ be r.i.~quasi-Banach function norms and $A$, $B$ their corresponding r.i.~quasi-Banach function spaces. Let $f\in \WL(A,B)$. Then $f\in\WL(A,B)_a$ if and only if
    \begin{equation*}
         \lim_{n\rightarrow \infty} \norm{f^*\chi_{[0,n^{-1})}}_{\overline{A}} = 0,
     \end{equation*}
     and
     \begin{equation*}
        \lim_{n\rightarrow \infty} \norm{f^*\chi_{[n,\infty)}}_{\overline{B}} = 0.
     \end{equation*}
\end{cor}
\begin{proof}
    This follows from Propositions~\ref{local_ri_equiv},~\ref{global_norm_rearrangement} and Theorem~\ref{amalgam_ac}.
\end{proof}

\bibliographystyle{dabbrv}
\bibliography{bibliography}
\end{document}